\documentclass[11pt]{amsart}

\usepackage[english]{babel}

\usepackage{color}
\usepackage{array}
\usepackage{enumerate}
\usepackage{enumitem}
\usepackage{refcount}
\usepackage{algorithm}
\usepackage{algorithmic}
\usepackage{boxedminipage}
\usepackage{float}
\usepackage[utf8x]{inputenc}
\usepackage{soul}

\usepackage{parskip}

 \usepackage[colorlinks]{hyperref}
\hypersetup{
    colorlinks,
    linkcolor={red!50!black},
    citecolor={green!50!black},
    urlcolor={blue!80!black}
}
\usepackage[nameinlink,capitalize]{cleveref}

\usepackage{color}
\usepackage{graphicx}
\usepackage{multirow}
\usepackage{tikz}
\usepackage{amsmath}
\usepackage{amssymb}
\usepackage{caption}
\usepackage{tikz-cd}

\newtheorem{thm}{Theorem}[section]
\newtheorem{prop}[thm]{Proposition}
\newtheorem{lem}[thm]{Lemma}

\theoremstyle{definition}
\newtheorem{defn}[thm]{Definition}
\newtheorem{nota}[thm]{Notation}
\newtheorem{exa}[thm]{Example}

\newtheorem{remark}[thm]{Remark}

\def\proof{{\bf Proof. }}

\newcommand{\bN}{\mathcal N}

\newcommand{\bL}{\mathcal L}
\newcommand{\bS}{\mathcal S}
\newcommand{\bM}{\mathcal M}

\newcommand{\prfend}{\hbox to7pt{\hfil}

\par\vskip-\baselineskip\hbox to\hsize
{\hfil\vbox {\hrule width6pt height6pt}}\vskip\baselineskip}

\def\qed{\rightline{$\rlap{$\sqcap$}\sqcup$}}

\def\a{\bigskip \par \noindent}
\def\b{\par \noindent}

\def\dd{\medskip \par \noindent}
\long\def\eatit#1{}

\def\N{\mathbb{N}}
\def\C{\mathbb{C}}

\def\P{\mathbb{P}}

\def\l{\lambda}
\def\m{\mu}
\def\e{\varepsilon}

\def\g{\gamma}

\font\tengothic=eufm10
\font\sevengothic=eufm7
\newfam\gothicfam
       \textfont\gothicfam=\tengothic
       \scriptfont\gothicfam=\sevengothic

\usepackage{xypic}

\begin{document}
\title{Some remarks on degeneracy of tridimensional tensors}

\author[A. Gimigliano]{Alessandro Gimigliano}
\author[M. Id\`a]{Monica Id\`a}
\address[A. Gimigliano, M. Id\`a]{Dipartimento di Matematica, Universit\`a di Bologna, Piazza di Porta S.Donato 5, 40126, Bologna, Italy}
\email{alessandr.gimigliano@unibo.it, monica.ida@unibo.it}

\maketitle

%%%%%%%%%%%%%%%%%%%%

\begin{abstract}

We study tridimensional tensors on $\C$ from the point of view of hypermatrices, taking into consideration the problem of determining whether they are degenerate or not, concise or not, and what is their essential format if they are non-coincise; moreover, in some cases, their tensor rank. We use a geometrical approach to these problems which,  in part, goes back to Schl\"{a}fli and consists in studying certain determinantal schemes associated to the hypermatrix.
\end{abstract}

\tableofcontents

\section{Introduction}

The study of hypermatrices goes back to Cayley, whose definition of hyperdeterminant had been mostly forgotten for a long time, until it has been deeply studied in the fundamental work \cite{GKZ}, now a classical reference on the subject (see also \cite{Gh}). Particularly, the tridimensional case has received a lot of attention (e.g. see \cite{K1}, \cite{K2}, \cite{F}, \cite{Ol}, \cite{TC}, \cite{T}, \cite{E}).

Here we consider tridimensional hypermatrices on $\C$ under various points of view (as associated to tensors, to multihomogeneous polynomials in three sets of variables, to trilinear applications), and we want to use a geometrical approach to their study which consists in considering three determinantal schemes defined by the matrices $L$, $M$ and $N$ of linear forms in one of the three sets of variables 	$x_i$, $y_j$ and $z_k$, associated to $A = (A_{ijk})$, for $i=1,\ldots, p$, $j=1,\ldots, q$, $k=1,\ldots, r$. This approach goes back to Sch\"{a}fli \cite{Sc} for particular formats, namely, in the tridimensional case, he considered hypermatrices where two among $p,q$ and $r$ are equal, so one of the matrices of linear forms is square, and he proved that the hyperdeterminant of $A$ divides the discriminant of the determinant of such square matrix.

We will consider, above all, the problem of how to determine when such $A$ is $degenerate$, i.e. when, denoting with $f_A$ the associated linear map, it exists a non-zero element $(a,b,c) \in \C^p\times \C^q\times \C^r$ such that
$$
f_A(\C^p, b,c) = f_A(a, \C^q, c) = f_A(a, b, \C^r) = 0 .
$$
Notice that for the formats when the hyperdeterminant $Det (A)$ is defined, $A$ is degenerate if and ony if $Det(A)= 0$ (see \cite{GKZ}).

We also consider other properties of $A$, as the tensor rank of the associated tensor and its  conciseness (i.e. whether the associated multi-homogeneous polynomial can be expressed using less variables). The notion of $concise$ $tensor$ and of $essential$ $format$ have been much studied for symmetric polynomials (see e.g. \cite{Ca}); in the general case these notions, which corresponds to the minimal numbers of variables that can be used to express the tensor, are what was classically called "index ranks" (e.g. see \cite{TC}).

It is worth noting that the study of a projective determinantal scheme is often 
computationally simpler than considering a tridimensional matrix (e.g. see \cite{A1}, \cite{A2}).

Our main reference for our study are \cite{GKZ}, the very nice expositive paper \cite{O},  and \cite{A1}, \cite{A2}, where the connection between degeneracy of $A$ and singularities or unexpected dimension of the associated schemes is studied. In section 3 we recall and reformat a result of \cite{A2} and its proof for the reader's convenience, and we explain why the main result in that paper is not true, so that our proposition \ref{A2Prop15} is a new result on this matter.

The structure of the paper is as follows: in section 2 we give preliminary definitions and properties that we need for this study; these are mostly well-known facts, but we list them here to set up notations and we prove some of them for the reader's convenience  (and particularly for who isn't acquainted with the matter); in section 3 we introduce Schl\"{a}fli's approach and its generalization, much on the line of what is done in \cite{T} and \cite{TC} to study the (3,3,3) and (3,3,4) formats in order to determine classes of tensors modulo linear changes of variables and permutations. We use this approach for our main results (Proposition \ref{A2Prop15} and Theorem \ref{Theorem}), which, as we said, give  a new corrected version of results in \cite{A2}. In section 4 we study hypermatrices of format $(2,2,r)$, $r\geq 2$ and, by using the associated determinantal schemes, we are able to give algorithms to determine when they are degenerate and when concise and to find their tensor rank and essential format.
\medskip

\section{Preliminaries}
In the following we refer to tridimensional hypermatrices (or tensors) on $\C$ and
give all definitions and properties for this case, even when they could be defined also for higher dimensional tensors (see also  \cite{O}). 
\begin{nota} \rm Let  $p, q, r$ be non-zero integers. We set $R:=\C[x_1,\ldots, x_{p},y_1,\ldots, y_{q},z_1,\ldots, z_{r}]$, and we denote by $R_{n,m,s}$ the set of $(n,m,s)$- homogeneous polynomials in $R$, i.e. the polynomials which are homogeneous of degree $n$ in $x_1,\ldots, x_{p}$, of degree $m$ in $y_1,\ldots, y_{q}$, and of degree $s$ in $z_1,\ldots, z_{r}$. 
\end{nota}

\begin{defn} \rm Let  $p, q, r$ be non-zero integers.
We call {\it $3$-dimensional hypermatrix  of format $(p, q, r)$}, or  {\it $(p, q, r)$-hypermatrix}  for short, a finite family indexed by three indeces \linebreak
$A = (a_{ijk})_{i=1,\dots,p; j=1,\dots,q; k=1,\dots,r }$.

If we fix basis $(\e_1,\dots,\e_p), (\eta_1,\dots,\eta_q),(\zeta_1,\dots,\zeta_r)$ respectively on the three vector spaces $\C^p , \C^q , \C^r$, and we denote the coordinates respectively by $(x_1,...,x_p), (y_1,...,y_{q}), (z_1,...,z_{r})$, we can associate biunivocally to $A$:
\begin{itemize} 
\item a tensor  in $\C^p \otimes \C^q \otimes \C^r$, namely the tensor $T_A=\sum_{i,j,k}a_{ijk}\, \e_i \otimes \eta_j \otimes \zeta_k$;
\item a trilinear function 
$$f_A : \C^p \times \C^q \times \C^r \to \C$$ 
$$(x_1,...,x_p,y_1,...,y_q,z_1,...,z_r)\mapsto \sum_{i,j,k}a_{ijk}x_iy_jz_k$$ 
(here $a_{ijk}= f_A(\e_i,\eta_j,\zeta_k)$);
\item the trihomogeneous polynomial $P_A := \sum_{i,j,k} a_{ijk}x_iy_jz_k \in R_{1,1,1}$.
\end{itemize} 

\par \medskip  We will consider $A$, $T_A$, $f_A$ and $P_A$ as defined modulo multiplication for a non-zero constant. The space $\P^m$, $m=pqr-1$, parameterizes $(p, q, r)$-hypermatrices, and hence also tensors in $\C^p \otimes \C^q \otimes \C^r$, trilinear maps $\C^p \otimes \C^q \otimes \C^r \to \C $ and trihomogeneous polynomials in $R_{1,1,1}$,  always modulo constants, while the Segre Variety $V_{p-1,q-1,r-1}\subset \P^m$  parameterizes the hypermatrices of tensor rank 1, i.e such that $P_A = \ell\ell '\ell ''$, where $\ell \in  \C[x_1,\ldots, x_{p}]_1$,\, $\ell ' \in  \C[y_1,\ldots, y_{q}]_1$ and $\ell '' \in  \C[z_1,\ldots, z_{r}]_1$. 

The $slices$ of the hypermatrix $A$ along the $x$-direction, respectively the $y$-direction, the $z$-direction ($x$-slices, $y$-slices and $z$-slices for short) are the matrices $A_{\bar{i}jk}$, respectively, $A_{i\bar{j}k}$, $A_{ij\bar{k}}$, defined as follows:

$A_{\bar{i}jk} = (a_{\bar{i}jk})$ ; $\qquad \bar{i}$ fixed in $\{1,\ldots,p\}$;

$A_{i\bar{j}k} = (a_{i\bar{j}k})$ ; $\qquad\bar{j}$ fixed in $ \{1,\ldots ,  q\}$;

$A_{ij\bar{k}} = (a_{ij\bar{k}})$ ; $\qquad\bar{k}$ fixed in $ \{1,\ldots ,  r\}$;

We say that two slices along the same direction are parallel to each other.

The $x$-rank of $A$ is the rank of the $x$-slices $A_{1jk}, \dots, A_{p,jk}$ as vectors in the space of matrices $q \times r$, and similarly for the $y$-rank and the $z$-rank.
\b A 0-$slice$ is a slice whose elements are all zeros.
\end{defn}

 \begin{remark}\label{rk} \rm Notice that, contrarily to the case of matrices, where the row-rank is equal to the column rank, for a tridimensional hypermatrix $x$-rank, $y$-rank and $z$-rank need not to be equal: for example, consider a $(2,2,3)$-hypermatrix, whose $x$-slices are 
 $$ \left( \begin{array} {cc}   1 &0 \cr 0&1 \end{array}\right), \; \left( \begin{array} {cc}   0 &1 \cr 1&0 \end{array}\right), \; \left( \begin{array} {cc}   1 &1 \cr 1&0 \end{array}\right).$$
 
 The $x$-rank is 3, while the $y$-rank and the $z$-rank are 2.
  \end{remark}
  
 \begin{defn}\label{Trk} \rm A tensor $t\in\C^p\otimes\C^q\otimes\C^r$ is said to be {\it decomposable} if there exist $u\in \C^p$, $v\in \C^q$, $w \in \C^r$ such that $t= u\otimes v\otimes w$. Given a tensor $T_A\in \C^p\otimes\C^q\otimes\C^r$, the {\it tensor rank} of $T_A$, denoted $Trk\, A$, is the minimal length of an expression of $T_A$ as a sum of decomposable tensors. 
 \end{defn} 
 
In this paper we use the following notations:

\begin{nota}\label{vectpt} \rm  If $(p_1,\dots, p_p)$ is a point of $\P^{p-1}$, with $P$ we denote both the point in $\P^{p-1}$ and the vector $\left( \begin{array} {cc}  p_1\cr \vdots \cr p_{p} \end{array}\right)$ in $\C^p$, and the same for $Q= (q_1,\dots,q_q)\in \P^{q-1}, T=(t_1,\dots,t_r)\in \P^{r-1}$. Hence $X= \left( \begin{array} {cc}  x_1\cr \vdots \cr x_{p} \end{array}\right)$ also denotes a variable point in $\P^{p-1}$, and the same for $\, Y= \left( \begin{array} {cc}  y_1\cr \vdots \cr y_q \end{array}\right)$ , $\,Z= \left( \begin{array} {cc}  z_1\cr \vdots \cr z_r \end{array}\right)$.
\end{nota}
\begin{lem}
A linear change of the coordinates $x_i$ in the space  $\C^p$, respectively $y_j$ in $ \C^q$, $z_k$ in  $\C^r$, corresponds to operating a linear combination on the $x$-slices, respectively $y$-slices, $z$-slices  of $A$. 
\b More precisely, if $A$ and $A'$ are the hypermatrices of the tensor in the coordinates $x,y,z$ and $x,y,z'$ respectively, and $Z= E\,Z'$  is the change of coordinates in $\C^r$, then the $z'$-slices of $A'$ are given by:

$$ \left( \begin{array} {cc}  A'_{ij1}\cr \vdots \cr A'_{ijr} \end{array}\right)\;=\; ^t\!E  \left( \begin{array} {cc}  A_{ij1}\cr \vdots \cr A_{ijr} \end{array}\right)$$

and analogously for $x$ and $y$.

\end{lem}
\proof
Let us consider a linear change of the coordinates $z_k$ on $\C^r$ (the proof is analogous for the other two sets of coordinates), where the new coordinates $z'_k$ are such that $z_k = \sum _{\lambda = 1} ^r e_{k\lambda}z'_\lambda$ and $E = (e_{k\lambda})$ is an invertible $r\times r$ matrix.
We have 
$$P_A = \sum_{ijk}a_{ijk}x_iy_jz_k =  \sum_{k=1}^r(\sum_{ij}a_{ijk}x_iy_j)\sum _{\lambda = 1}^r e_{k\lambda} z'_\lambda = \sum_{ij\lambda}(\sum_{k =1}^r a_{ijk}e_{k\lambda})x_iy_jz'_\lambda$$
Hence the hypermatrix of the coefficients of $P_A$ with respect to the coordinates $x_i$, $y_j$, $z'_\l$ is  \linebreak $A' = (a'_{ij\l})$, with $a'_{ij\lambda}= \sum_{k=1}^ra_{ijk}e_{k\lambda}$.  So for any fixed $\bar \l$ the $z'$-slice of $A'$ is 
$$ A'_{ij\bar \l}=(a'_{ij\bar \l})= e_{1\bar \l}A_{ij1}+\dots+ e_{r\bar \l}A_{ijr}.$$

\prfend

\begin{remark} \rm
Notice that the notions we are interested in for $A$ and $f_A$ (such as degeneracy, essential format) are invariant respect to the operations above (e.g. see \cite{O}).
\end{remark}

\subsection{Conciseness}

\begin{defn}\label{concise}\rm We say that the hypermatrix $A$ is $non$-$concise$ if, modulo linear changes of coordinates, $P_A$ can be written using less coordinates, and $concise$ if this does not happen; for example, when a ``slice" of $A$ is a 0-matrix, $A$ is non-concise. More precisely, $A$ is non-concise if there are linear forms $x'_1,...,x'_{p'} \in  \C[x_1,\ldots, x_{p}]$,  $y'_1,\ldots, y'_{q'} \in  \C[y_1,\ldots, y_{q}]$ and $z'_1,\ldots, z'_{r'} \in  \C[z_1,\ldots, z_{r}]$, such that the following inequalities hold: $p'\leq p$; $q'\leq q$; $r'\leq r$, and at least one of these is strict. Then we have 
  $$P_A\in \C[x'_1,\ldots, x'_{p'}; \,y'_1,\ldots, y'_{q'}; \, z'_1,\ldots, z'_{r'}]_{1,1,1}.$$
 If this happens, we say that $A$, or $P_A$, can be reduced to the format $(p',q',r')$.

  \bigskip \par Now we want to show that there exists a minimum for the formats $(p',q',r')$ to which $A$ can be reduced, and the forthcoming \ref{coincise} proves it.
This allow us to define the $essential$ $format$ of $A$,  denoted $F_{ess}(A)$, as the minimum format $(p',q',r')$, such that $P_A$ can be written as above; in other words, $A$ is concise if and only if $F_{ess}(A)= (p,q,r)$.
 \end{defn}
 
We prefer the denomination above (essential format) rather than the older one of "index ranks" (e.g. see [T]), not to be confused with "tensor rank".

 \begin{prop}\label{coincise} A $(p,q,r)$ hypermatrix $A$ is non-concise if and only if via linear combinations of its slices we can reduce it in a form containing at least a 0-slice. 
\b More precisely, if $F_{ess}(A) = (p',q',r')$, we have $p'=x$-rank, $q'= y$-rank, $r'=z$-rank, that is, via linear combinations of its slices we can obtain $p-p'$, $q-q'$ and $r-r'$ 0-slices in the three directions, respectively.
\end{prop}
\proof
Setting  $X= \left( \begin{array} {cc}  x_1\cr \vdots \cr x_{p} \end{array}\right)$ and $X'= \left( \begin{array} {cc}  x'_1\cr \vdots \cr x'_{p} \end{array}\right)$, consider the change of coordinates in $\C^p$ given by $X=CX'$ where $C= \left( \begin{array} {ccc}  c_{11}& \dots & c_{1,p}\cr \vdots & \vdots & \vdots \cr c_{p,1}& \dots & c_{p,p} \end{array}\right)$ is an invertible matrix, and analogously $Y=DY'$, $Z=EZ'$. We have:

$$f_A (^t\!X, \, ^t\!Y,\, ^t\!Z)=  \sum_{i,j,k}a_{ijk}x_iy_jz_k= (^t\!Y A_{1jk} Z)x_1+\dots + (^t\!YA_{pjk}Z)x_{p}=$$
$$=  \left( \begin{array} {ccc}  ^t\!Y A_{1jk} Z &\dots & ^t\!YA_{pjk}Z \end{array}\right) \; X=$$
$$=\left( \begin{array} {ccc}  ^t\!(DY')\,A_{1jk}\,EZ' &\dots & ^t\!(DY')\,A_{pjk}\,EZ'  \end{array}\right) \; CX'=$$
$$= \left( \begin{array} {ccc} ^t\!Y'(^t\!D\,A_{1jk}\,E)Z'  &\dots & ^t\!Y'(^t\!D\,A_{pjk}\,E)Z'   \end{array}\right) \; CX'$$

Now consider, for example, the variable $x_{p}$; it does not appear in this writing if and only if $ \left( \begin{array} {ccc} ^t\!Y'(^t\!D\,A_{1jk}\,E)Z'  &\dots & ^t\!Y'(^t\!D\,A_{pjk}\,E)Z'   \end{array}\right) \; \left( \begin{array} {cc}  c_{1,p}\cr \vdots \cr c_{p,p} \end{array}\right)=0$, that is, iff 
$$^t\!Y'( c_{1,p}\,^t\!D\,A_{1jk}\,E + \dots + c_{p,p}\,^t\!D\,A_{pjk}\,E)Z'=0 \Leftrightarrow c_{1,p}\,^t\!D\,A_{1jk}\,E + \dots + c_{p,p}\,^t\!D\,A_{pjk}\,E=0 $$
$$\Leftrightarrow \; ^t\!D\,(c_{1,p}A_{1jk}+ \dots + c_{p,p}A_{pjk})\,E =0 $$
$$\Leftrightarrow c_{1,p}A_{1jk}+ \dots + c_{p,p}A_{pjk}=0 $$
Hence we can make an $x$-variable (in this example $x_p$) disappear in the writing of $f_A$ if and only if the $x$-slices of $A$ are linearly dependent. Going on, we see that it is possible to use, for $f_A$, exactly $p'$ of the $x$-variables, and not less, if and only if the $x$-rank of $A$ is $p'$.
\qed

  \begin{remark} \rm Let us recall what happens in the trivial cases of linear and bilinear maps, which is no longer true for trilinear maps.
  \par 1) Let $f:V \to \C$ be a linear map and  $\e_1,\dots, \e_n$ a basis of $V$ such that $Ker f=<e_{k+1},\dots, \e_n>$; then $f$ can be written using $k=dim V - dim Ker f$ variables, i. e. the essential format of $f$ is $rk f$.
  \par 2) Let $f: V\times W \to \C$ be a bilinear map, $f(X ,Y)=\,^t\!X A Y$, with $dim V=p, dim W= q$, and fix bases $(\e_1,\dots,\e_p), (\eta_1,\dots,\eta_q)$. The linear map $g:W \to V$, $Y \mapsto AY$, with a suitable change of basis $Y=DY'$, will have an associated matrix with $q-rk A$ 0-columns (and not more), so $f$ can be written as $f(X ,Y')=\,^t\!X (AD) Y$ and $y_{rk A+1}, \dots, y_q$ have disappeared (notice that if $rk A= q$, all $y_j$ have to stay there). Now let $B:= AD$ and consider the linear map $h: V \to W$, $X\mapsto \,^t\!BX$; as before, there is a linear change of coordinates $X=CX'$ such that $^t\!B C$ has the maximum number of 0-columns, i.e. $p-rk (B)=p-rkA$. Hence $^t\!C B$ = $^t\!C AD  $ has the maximum number of 0- rows.
  In conclusion, the essential format of a bilinear map is $(rk A, rk A)$, i.e. a bilinear map is coincise iff $p=q=rkA$.
  \end{remark}

\medskip
  \begin{remark}\label{concisealg} \rm   If $r > p q$, $A$ is never concise; in fact, its $z$-slices are $r$ matrices $p\times q$, hence at least $r-pq $ are linearly dependent from the others. 
  \b Another way to see this is the following:
  
  An immediate algorithm for conciseness works via flattenings (see \cite{O}): $A$ is concise if and only if its three flattenings, which are the following matrices of format $(p,q\times r)$, $(q,p\times r)$, $(p\times q, r)$:
  $$B = \left( \begin{array} {ccc} a_{111} & \dots & a_{1qr}\cr \dots & \dots & \dots \cr a_{p11} & \dots & a_{pqr}\end{array}\right), C = \left( \begin{array} {ccc} a_{111} & \dots & a_{p1r}\cr \dots & \dots & \dots \cr a_{1q1} & \dots & a_{pqr}\end{array}\right), D = \left( \begin{array} {ccc} a_{111} & \dots & a_{pq1}\cr \dots & \dots & \dots \cr a_{11r} & \dots & a_{pqr}\end{array}\right)$$ have rank $p,q,r$ respectively. This implies that if $r > p q$, $A$ is never concise. 
\end{remark}

\medskip
  \begin{remark}\label{conciseclosed}
Let us notice that the subset of non-concise hypermatrices is closed since it corresponds either to the whole space or to polynomials which can be written by using a smallest set of variables.
\end{remark}

\subsection{Determinantal schemes}

 In the following with determinantal scheme we mean a scheme $\mathcal S$ whose ideal is generated by the maximal minors of a matrix $B$ of linear forms (hence as a set it is given by the points in which the matrix has not maximal rank). Namely, let  $u\leq v$ and $B$ be a $(u \times v)$-matrix of linear forms in $n$ variables $x_1, \dots, x_n$; so the maximal minors are the determinants of the submatrices $u \times u$, and $\bS \subseteq \P^{n-1}$. We say that the scheme $\bS$ is associated to the matrix $B$.
\par \medskip
 Since $\emptyset \subseteq S \subseteq \P^{n-1}$, one has $0\leq codim \bS \leq n$. The expected codimension of $\bS$ is  (see \cite{BV})
$$\exp codim \bS =\min \{v-u+1, n\}.$$

This is actually the codimension of $\bS$ if the matrix $B$ is generic,  i.e. in the space $\P^{uv-1}$ parameterizing $u\times v$ matrices with coordinates $(w_{ij})$, $i\in \{1,\ldots u\}$, $j\in \{1,\ldots v\}$, the scheme $\bS$ corresponds to the intersection between the variety defined by the $u\times u$ minors of $(w_{ij})$, whose codimension is known to be $v-u+1$ (e.g. see  \cite{ACGH}), with the generic $(n-1)$-dimensional projective subspace whose parametric equations are given by the linear forms which are the entries of $B$.

If $\bS$ fails to have the expected codimension, then  $codim \bS <\min \{v-u+1, n\}$ (see \cite{BV}).
\b Let $P$ be a point of $\bS$; the Jacobian matrix of $\bS$ at $P$ is a ${v\choose u} \times n$-matrix , and in the following it will be denoted by $J_\bS(P)$.
\b The codimension of the tangent space $T_P(\bS)$ of $\bS$ at $P$, which is the rank of  $J_\bS(P)$,  is equal to the codimension of $\bS$ at $P$, unless $P$ is singular for $\bS$.  We say that the expected rank for $J_\bS(P)$ is 
$$ \exp {\rm rk}\, J_\bS(P)=\min \{v-u+1, n\}$$ 

\begin{defn}\label{determ} We say that $P\in \bS$ is a {\it degenerate point} for $\bS$ if ${\rm rk}\, J_\bS(P)<\min \{v-u+1, n\}$; hence $P$ is degenerate if either $\bS$ has the expected dimension at $P$ and $P$ is singular, or the dimension of $\bS$ at $P$ is greater than the expected one (in which case $P$ can be either a smooth or a singular point).
\par  We say that $\bS$ is {\it degenerate} if it has at least one degenerate point.
\par  We say that a point $P\in \bS$ is {\it bi-degenerate} if all the $(u-1)\times (u-1)$-minors of $B$ vanish at $P$. Since such a point is singular for $\bS$ (see \cite{BV}), any bi-degenerate point is also a degenerate point.
\end{defn}

\begin{remark}\label{sqmatrix} The set of bi-degenerate points has codimension $\leq 2(v-u+2)$ (\cite{BV}, Theorem 2.1), so, for example, when $u=v$, $\bS$ always possesses bi-degenerate points if $n\geq 5$.
\end{remark}

\begin{remark}\label{jacobfunct} Let  $\bS \subseteq \P^{n-1}$  be a determinantal scheme, defined by a $u \times v$ ($u\leq v$)  matrix $B$ of linear forms: 
$$ B= \left(\begin{array} {ccc} B_{11} & \ldots & B_{1v}\cr  & \ldots &  \cr B_{u1}&\ldots & B_{uv} \end{array} \right), \quad B_{ij}= \sum_{k=1}^n a_{ijk}x_k$$
In the following $\alpha $ denotes a multi-index: $\alpha = (j_1,\dots, j_u)$, with $1\leq j_1 \leq \dots \leq j_u \leq v$, and we shall always assume that an order is chosen in the set of these multi-indexes, so we can write $\alpha =1, \dots, {v\choose u}$. 
\b Setting $$B_{\alpha} :=  \left|\begin{array} {ccc} B_{1j_1} & \ldots & B_{1j_u}\cr  & \ldots &  \cr B_{uj_1}&\ldots & B_{uj_u} \end{array} \right|, \quad \alpha= (j_1,\dots, j_u)$$

\a we have $I_X=(B_1, \dots, B_{v\choose u})$, and the Jacobian matrix of $X$ is given by 
$$ J(\bS)= (\partial_k B_\alpha) =\left(\begin{array} {ccc} \partial_1 B_1 & \ldots & \partial_n B_1\cr  & \ldots &  \cr \partial_1 B_{v\choose u}&\ldots & \partial_n B_{v\choose u} \end{array} \right).$$

If $f_{11}, \dots, f_{uu}$ are functions of $x_1, \dots, x_n$, it is easy to see, for example by induction, that 
$$\partial_k \left|\begin{array} {ccc} f_{11} & \ldots & f_{1u}\cr  f_{21} & \ldots & f_{2u}\cr & \ldots &  \cr f_{u1}&\ldots & f_{uu} \end{array} \right| = 
\left|\begin{array} {ccc} \partial_k f_{11} & \ldots & \partial_k f_{1u}\cr  f_{21} & \ldots & f_{2u}\cr & \ldots &  \cr f_{u1}&\ldots & f_{uu} \end{array} \right|  +  \left|\begin{array} {ccc} f_{11} & \ldots & f_{1u}\cr  \partial_k f_{21} & \ldots & \partial_k f_{2u}\cr & \ldots &  \cr f_{u1}&\ldots & f_{uu} \end{array} \right|   +
\dots + \left|\begin{array} {ccc} f_{11} & \ldots & f_{1u}\cr  f_{21} & \ldots & f_{2u}\cr & \ldots &  \cr \partial_k f_{u1}&\ldots & \partial_k f_{uu} \end{array} \right| .$$

Hence for any $\alpha$ we have: $$ \partial_k B_\alpha=\left|\begin{array} {ccc} a_{1j_1k} & \ldots & a_{1j_uk}\cr  B_{2j_1} & \ldots & B_{2j_u}\cr & \ldots &  \cr B_{uj_1}&\ldots & B_{uj_u} \end{array} \right| +
  \left|\begin{array} {ccc} B_{1j_1} & \ldots & B_{1j_u}\cr  a_{2j_1k} & \ldots & a_{2j_uk}\cr & \ldots &  \cr B_{uj_1}&\ldots & B_{uj_u} \end{array} \right| +\dots + 
  \left|\begin{array} {ccc} B_{1j_1} & \ldots & B_{1j_u}\cr  B_{2j_1} & \ldots & B_{2j_u}\cr & \ldots &  \cr a_{uj_1k}&\ldots & a_{uj_uk} \end{array} \right| 
  \eqno{(\star)}$$
\a Now consider the point $P=(1,0,\dots,0) \in \P^{n-1}$; we have $B_{ij}(P)= a_{ij1}$, hence if we assume $a_{1j1}=0$ for $j=1,\dots, v$, we have $B_{1j}(P)= a_{1j1}=0$ and $(\star)$ becomes:
$$\partial_1 B_\alpha (P)=0, \qquad  \partial_k B_\alpha (P)=\left|\begin{array} {ccc} a_{1j_1k} & \ldots & a_{1j_uk}\cr  a_{2j_11} & \ldots & a_{2j_u1}\cr & \ldots &  \cr a_{uj_11}&\ldots & a_{uj_u1} \end{array} \right| {\; \rm for} \: k=2,\dots, n, {\; \rm and\; for \; any\;} \alpha    \eqno{(\heartsuit)}$$
\end{remark}

\subsection{Degeneracy and associated determinantal schemes}

\par We give here some definitions for the tridimensional case following the general definitions for hypermatrices (see \cite{GKZ}). In the following, as always, $A$ is a hypermatrix of format $(p, q, r)$, with $2\leq p\leq q \leq r$; recall notations \ref{vectpt}.
\begin{defn} \rm 
To a 3-dimensional hypermatrix $A$ we associate three matrices of linear forms, respectively in
$\C[x_1,\ldots , x_p]$, $\C[y_1,\ldots , y_q]$, $\C[z_1,\ldots , z_r]$, in the following way:
\begin{itemize}
\item $ L = (L_{jk})$ is a $q \times r$ matrix, where $L_{jk} =  \sum^p_{i=1} a_{ijk}x_i$ ;
\item $M = (M_{ik})$ is a $p \times r$ matrix, where $M_{ik} =  \sum^q_{j=1} a_{ijk}y_j$ ;
\item $N = (N_{ij})$ is a $p \times q$ matrix, where $N_{ij} = \sum^r_{k=1} a_{ijk}z_k$ .
\end{itemize}
\par If $P\in \P^{p-1}$, with $L(P)$ we denote the matrix $ L (P)= (L_{jk}(P))$, and the same for $M$ and $N$.
\medskip \par 

The matrices $L$, $M$, $N$ define three determinantal projective  schemes, respectively $\bL \subset \P ^{p-1}$,  $\bM\subset \P ^{q-1}$, $\bN\subset \P ^{r-1}$. 
 We say that $L$, $M$, $N$ are the $associated$ $matrices$ to the hypermatrix $A$, as well as to the tensor $T_A$, and $\bL $, $\bM$, $\bN$ are the $associated$ $schemes$.
\end{defn}

\medskip
\begin{defn}\label{kernel} \rm The $ kerne$l $K_A$ or $K_{f_A}$ of a trilinear form $f_A$ associated to $A$ is the set of all
triplets of points $(P,Q,T)\in \P^{p-1} \times \P^{q-1 }\times \P^{r-1 }$ whose coordinates satisfy the three following conditions:

$\quad  f_A(P,Q,Z) = 0  \quad \quad \forall Z \in \P^{r-1}$,

$ \quad f_A(P,Y,T) = 0  \quad \quad \forall Y \in \P^{q-1} $,

$ \quad f_A(X,Q,T) = 0  \quad \quad \forall X \in \P^{p-1}$. 

\a It is immediate to see that the kernel $K_A$ is the set of all
triplets of points $(P,Q,T)\in\ \P^{p-1} \times \P^{q-1 }\times \P^{r-1 }$ whose coordinates satisfy the system  of $p+q+r$ equations:
$$(*)\qquad \qquad \qquad
\left\{ \begin{array}{l}
\sum_{i, j}a_{ijk}x_iy_j = 0;   \quad  k= 1,\ldots, r \qquad \qquad{(*1)}\\
\\
\sum_{ i, k}a_{ijk}x_iz_k = 0;   \quad  j= 1,\ldots, q \qquad \qquad{(*2)}\\
\\
\sum_{j, k}a_{ijk}y_jz_k = 0;   \quad   i= 1,\ldots, p \qquad \qquad{(*3)}. \end{array}\right.
$$

We say that a hypermatrix $A$ (or the map $f_A$) is $degenerate$ if $K_A \neq \emptyset$. \end{defn}

\begin{remark}\label{matid} \rm In the previous notations it is immediate to see that the following matricial identities hold: 

$$^t\!Y\,L = ^t\!\!X\,M,  \qquad  \qquad ^t\!X\,N = ^t\!\!(L\,Z), \qquad  \qquad M\,Z = N\,Y$$

It is also immediate to see that we can write system $(*)$ as 
$$(*)\qquad 
\left\{ \begin{array}{l}
^t\!Y\,L=0 \qquad \qquad{(*1)}\\
\\
^t\!\!X\,N=0 \qquad \qquad {(*2)}\\
\\
M\,Z=0 \qquad  \qquad {(*3)}. \end{array}\right.
$$
\end{remark}

\begin{prop}\label{notemp} In the previous notations we have that 
\dd i) The existence of a point in $\bL$ implies the existence of a point  in $\bM$ and viceversa. 

More precisely, if $P\in \bL$ and a non-trivial relation among the rows of $L(P)$ is given by $ ^t\!Q\,L (P)=0$, then $Q\in \bM$ and a non-trivial relation among the rows of $M(Q)$ is given by $^t\!\!P\,M(Q)=0$; and viceversa. 
\dd ii)  If the system given by $(*1)$ and $(*2)$ has a solution $(P,Q,T)$, then $P\in \bL, \; Q\in \bM, \; T\in \bN$; in particular if $A$ is degenerate its associated schemes $\bL$, $\mathcal M$ and $\mathcal N$ are not empty. 
\end{prop}
\proof Saying that $ ^t\!Q\,L (P)=0$ means that $Q$ gives a null linear combination of the rows of $L(P)$, and since $p\leq q$ this implies $P\in \bL$; analougously for $\bM$ and $\bN$.

{\it i)} We have $ ^t\!Q\,L (P)= ^t\!\!P\,M(Q)$ (see \ref{matid}), hence  $ ^t\!Q\,L (P)=0 \iff  ^t\!\!P\,M(Q)=0$, i.e. $P\in \bL \iff Q\in \bM$.

\dd {\it ii)} If  the subsystem $(*1) + (*2)$ has a solution $X=P$, $Y=Q$, $Z=T$, then we have 
$$ ^t\!Q\,L (P)= ^t\!\!P\,M(Q)=0, \quad ^t\!\!P\,N(T)=0 .$$
Hence $P\in \bL$, $Q\in \bM$ and $T\in \bN$. 

\prfend

\begin{remark}\label{elop}
Let $B$ be a $q \times r$ matrix; the substitution of one of its rows $R$ (respectively column) with a linear combination of the rows (respectively columns) where $R$ appears with a non-zero coefficient, is called an elementary operation on the rows, respectively columns. This  gives a new matrix $B'$, respectively $B''$, which can be seen as $B'=E\,B$, respectively $B''=B\,E$ where $E$ is a suitable invertible $r\times r$, respectively $q\times q$, matrix.

Given a hypermatrix $A$, the matrix $L$ is $L = (\sum_{i=1}^pa_{ijk}x_i)$, hence we have $P_A =  ^t\!Y\,L\,Z$. If we make an elementary operation on the rows of $L$, we get $L'=E\, L$ with $E$ invertible, and this gives  a new hypermatrix $A'$ with 
$$P_{A'}= ^t\!Y\,L'\,Z = ^t\!Y\,(E\, L)\,Z= (^t\!Y\,E)\, L\,Z= ^t\!Y'\,L\,Z$$
where $Y'=^t\!E\,Y$, i. e. we are considering the same tensor $A$ after a change of coordinates in the $y_j$'s. Clearly, if we operate on the columns of $L$ we are making a change of coordinates in the $z_k$'s. Reasoning analogously for $M$ and $N$, we see that an elementary operation on the rows or on the columns of $L$, respectively $M$,  $N$, gives the same tensor modulo a change of coordinates not involving the $x_i$'s, respectively the $y_j$'s, the $z_k$'s.
\end{remark}

We conclude this section with the following
\begin{prop}
If a hypermatrix $A$ is non-concise, then its three associated  schemes are degenerate.
\end{prop}
\proof By Remark \ref{concise} we have that (modulo linear changes of the coordinates $x_i$, $y_j$ and $z_k$, i.e. modulo linear combinations on the slices of $A$) there is at least a 0-slice in one of the three directions. Say the 0-slice is the last one in the $r$ direction, i.e. $(a_{ijr}) = \mathbf{0}$. We have that in this case the associated matrices $L$ and $M$ have a column of zeroes, hence the codimensions of $\bL$ and $\bM$ will be at least one less then expected, hence they are degenerate.  In the matrix $N$, instead, the entries will be polynomials in $z_1,\ldots , z_{r-1}$, so the scheme $\bN$ is a cone and it has singular points (at least $(0,\ldots, 0,1)$), hence it is degenerate. 
\prfend

\subsection{Hyperdeterminant}

 We reformulate here the definition of hyperdeterminant for a tridimensional hypermatrix
(for the general definition, see \cite{GKZ}).

Let $p$, $q$, $r$ be integers, with $2\leq p \leq q \leq r$. Consider the variety $V_{p-1,q-1,r-1}$ given by the product $\P^{p-1} \times \P^{q-1} \times \P^{r-1}$ in its Segre embedding in $\P^m$, where $m=pqr-1$ (defined by the linear system given by forms of multidegree $(1,1,1)$). Let $\C[w_{ijk}] $ be the coordinate ring of $\P^m$.
\b We denote the dual variety of $V_{p-1,q-1,r-1}$ with $\mathbb{V}$; recall that $\mathbb{V} \subset (\P^m)^\ast$ is the projective closure of the set of all hyperplanes in $\P^m$ that are tangent to $V_{p-1,q-1,r-1}$ (i.e. that contain the tangent space at a point of $V_{p-1,q-1,r-1}$). 
\b Let $A = (a_{ijk})$ be a hypermatrix of format $(p, q, r)$. Then $H_A: \sum_{i,j,k} a_{ijk}w_{ijk} = 0$ is a hyperplane in $\P^m$, that is, a point in $(\P^m)^\ast$.

\begin{defn}  \rm In the previous notations assume that $p$, $q$, $r$ are such that $\mathbb{V}$ is a hypersurface of $(\P^m)^*$, with $I_{\mathbb{V}} =(D)$. Then
\b $i)$ $D$  is called the {\em hyperdeterminant of format }$(p, q, r)$;
\b $ ii)$  we write $Det(A) := D(H_A)$ (obviously the value of $Det(A)$ is defined only when it is zero).
 
 Notice that the hyperdeterminant is not defined for those formats for which codim$(\mathbb{V})>1$.
\end{defn}

As a corollary of \cite{GKZ}, Chapter 14, Theorem 1.3, we have the following:

\begin{prop}\label{format} A tridimensional hyperdeterminant of format $(p, q, r)$ ($2\leq p \leq q \leq r$), is defined if and only if
$$ r\leq p + q -1.$$
\end{prop}

\begin{defn} {\rm We say that a format $(p, q, r)$, with $2\leq p \leq q \leq r$, is $boundary$ if the inequality in Proposition \ref{format} is an equality. If the inequality holds strictly we say that the format is $interior$, and that it is $exterior$ if $ r > p + q -1.$}
\end{defn}

Always from \cite{GKZ}, we have:
\begin{prop}\label{degenHyperdet} A tridimensional hypermatrix of format $(p, q, r)$ ($2\leq p \leq q \leq r$), with $ r\leq p + q -1$ is degenerate if and only if $Det(A) = 0$.
\end{prop}

Since the condition $Det(A) = 0$ means that the point $H_A$ is on the hypersurface $\mathbb{V}$, \ref{degenHyperdet} implies that the generic tensor of format $(p, q, r)$ in the interior or boundary case is not degenerate.

Notice that a statement analogous to \ref{degenHyperdet} holds for hypermatrices of any dimensions and format whenever the hyperdeterminant is defined.

\bigskip

\section{ Schl\"{a}fli's approach and a generalization}

The approach to the study of  the hyperdeterminant of $A$ via the associated schemes is originally due to Ludwig Schl\"{a}fli (1814-1895, \cite{Sc}). Schl\"{a}fli found that the hyperdeterminant of a hypermatrix $A$ (of any dimension) divides the discriminant of the hyperdeterminant of one of its associated hypermatrices if such hyperdeterminant is defined (see \cite{GKZ}, Ch.14, Thm 4.1). In our case, since $A$ is trilinear, those are the matrices of linear forms $L, M, N$; hence in our case such hyperdeterminant is a determinant which exists only for square matrices,
i.e.  when we consider formats $(p,q,q)$ or $(p,p,r)$. In particular  we have (notations as above; see \cite{O}):

\begin{thm}\label{simpleformats} For hypermatrices $A$ of format $(2, q, q)$ and $(3, q, q)$ (in the second case we have $q\geq 3$)  $L$  is a square matrix, and  the discriminant of $\det(L)$ coincides with the hyperdeterminant of $A$, hence $A$ is degenerate if and only if the scheme $\bL$ is degenerate. 
\end{thm}

\begin{remark}
For format $(2, 2, 3)$ the same is true for $\bN$, i.e. $Det(A)$ is the discriminant of the conic $\bN$ (see \cite{O}, Example 5.4).
\end{remark}

While Schl\"afli method uses only one associated scheme, our aim is to use all the three associated schemes to get informations about $A$, mainly about its degeneracy. For example it is known, when $A$ is of boundary format, that $A$ is degenerate if and only if the scheme $\bL$ (or, equivalently, the scheme $\bM$) possesses a degenerate point (again, see \cite{GKZ}); here we try to exploit the relation between degeneracy of $A$ and degenerate points of $\bL$, $\bM$ and $\bN$. 

The work in \cite{A1} and \cite{A2} is aimed exactly to this target, but unfortunately, as we noticed in the introduction, the main result there, namely Theorem 10 of \cite{A2} which affirms the equivalence of $A$ being degenerate (when  $r\leq p+q-1$) and the existence of a degenerate point in anyone of the three schemes $\bL$, $\bM$ and $\bN$ is wrong, as the following counterexample shows; hence a more subtle relation must be discovered.

\begin{exa}\label{335} Let $A$ be a generic $(3,3,5)$ hypermatrix; being generic $A$ is not degenerate, while the scheme $\bN\subset \P^4$ necessarily has bi-degenerate points, since the subscheme $\bN'\subset \bN\subset \P^4$ whose ideal is given by the $2\times 2$ minors of the $3\times 3$-matrix $N$ has codimension 4 (see \ref{sqmatrix}) hence it is not empty.
\end{exa}

\a The following Proposition \ref{A2P14} and Lemma \ref{A2L11} are stated in \cite{A2} (Proposition 14 and Lemma 11) in the interior and boundary cases. Here the idea of the proof is  the same, but we give a more compact proof which is valid for all formats.

\begin{prop}\label{A2P14} Let $A=(a_{ijk})$ be a tridimensional hypermatrix of format $(p,q,r)$ with $2\leq p\leq q \leq r$. If $A$ is degenerate, let $(P,Q,T)\in K_A$, then the three points $P,Q$ and $T$ are degenerate points of $\bL, \bM$ and $\bN$, respectivey.
\end{prop}
\proof  
We start with the scheme $\bL$. We want to check that $P\in \bL$ is degenerate, i.e. that $$rk\, J_\bL(P) <  \exp {\rm rk}\, J_\bS(P)=\min \{r-q+1, p\}.$$
If  $r-q+1 \geq p$ we are in the boundary and exterior cases, and we have to prove that $rk\,J_\bL(P) \leq p-1$, or, which is the same, that $\dim \bL\geq 0$ (the expected dimension of $\bL$ in this case is $-1$, i.e. it is expected to be empty). Hence the mere existence of $P$ in $\bL$ (see Proposition \ref{notemp}) makes it a degenerate point.
\b If $r-q+1 < p$, that is, if $r< p+q-1$, we are in the interior case and we have to prove that $rk\,J_\bL(P) \leq r-q$.

\b Let $(P,Q,T) \in K_A$, and choose coordinates such that  $P=(1,0,\ldots 0)$, $Q = (1,0,\ldots 0)$ and $T=(1,0,\ldots 0)$. With this choice, by  $(*1), (*2)$ and $(*3)$ we have :
$$ a_{11k}=0 \quad {\rm for} \; k=1,\ldots , r;  \qquad  a_{1j1}=0 \quad {\rm for} \; j=1,\ldots , q;  \qquad a_{i11}=0 \quad {\rm for} \; i=1,\ldots p \qquad \eqno{(\diamondsuit)}$$

Hence we are in the conditions to apply $(\heartsuit)$ in \ref{jacobfunct} and we get:
$$\partial_1 L_\alpha (P)=0, \qquad \partial_iL_{\alpha}(P) =  \left| \begin{array} {cccc} a_{i1k_1}&\ldots&a_{i1k_q}\cr a_{12k_1}&\ldots&a_{12k_q}\cr \vdots &\ldots &\vdots \cr a_{1qk_1}&\ldots&a_{1qk_q}  \end{array} \right| {\; \rm for} \: i=2,\dots, p.$$

It follows that the entries of the $i^{th}$ row of $J_\bL(P)$ are the $q\times q$ minors of the following matrix:

$$J_i = \left(\begin{array} {cccc} a_{i11}& a_{i12}&\ldots  a_{i1r}\cr a_{121}& a_{i22}&\ldots  a_{12r}\cr \vdots&\vdots&\ldots \vdots \cr a_{1q1}& a_{1q2}&\ldots   a_{1qr}\end{array} \right) = \left(\begin{array} {cccc} 0& a_{i12}&\ldots   a_{i1r}\cr 0& a_{122}&\ldots   a_{12r}\cr \vdots&\vdots&\ldots  \vdots \cr 0& a_{1q2}&\ldots  a_{1qr}\end{array} \right).$$

\a  The inequality $rk\,J_\bL(P) \leq r-q$ follows by applying point 2) of Lemma \ref{A2L11} below to the $q\times (r-1)$ matrices $\tilde J_i$ (where $\tilde I_i$ is $J_i$ with the first column of zeros deleted), whose $q\times q$ minors give the Jacobian matrix at $P$.

\a The role of $p$ and $q$ being symmetric, exactly the same discussion holds for $\bM$.

\dd In the case of the scheme $\bN$ we want to prove that $T$ is a degenerate point, i.e. that $rk\, J_\bN(T) <  \min \{q-p+1, r\}.$ Since by assumption $2\leq p \leq q \leq r$, we have $q-p+1< r$, hence we have to prove that $rk\, J_\bN(T)) \leq q-p$; this can be done exactly as we have done for $\bL$ in the interior case.
\prfend

\begin{lem}\label{A2L11} 
1) Let $D$ be the $m\times n$ matrix ($m\leq n$):
$$D = \left(\begin{array} {ccc} x_1 & \ldots & x_n\cr a_{21}&\ldots & a_{2n}\cr \vdots & \ldots & \vdots \cr a_{m1}&\ldots & a_{mn} \end{array} \right), $$
where the $a_{ji}'s \in \C$ and the $x_i's$ are indeterminates. Then the number of linearly independent maximal minors of $X$ is $\leq n-m+1$.

\medskip
2) Let $D^{(1)}, \ldots , D^{(t)}$ be the $m\times n$-matrices with entries in $\C$ obtained setting $x_1=c_1^{(\ell)}, \dots, x_n=c_n^{(\ell)}$ in $D$:
$$D^{(\ell)} = \left(\begin{array} {ccc} c_1^{(\ell)} & \ldots & c_n^{(\ell)}\cr a_{21}&\ldots & a_{2n}\cr \vdots & \ldots & \vdots \cr a_{m1}&\ldots & a_{mn} \end{array} \right), \quad \ell =1,\ldots t.$$

\medskip Let $\alpha = (k_1,\ldots,k_m) \in \N^m$, with $1\leq k_1<k_2<\ldots < k_m\leq n$ and $B_{\ell \alpha}$ the $m\times m$-minor of $D^{(\ell)}$ given by the columns $k_1,\ldots , k_m$. Let  $B$ the $t\times {n \choose m}$-matrix whose entries are the $B_{\ell \alpha}$ for $\ell =1,\ldots,t$ and $\alpha$ as above.  Then $rk (B) \leq n-m+1$. 
\end{lem}

\proof 
1) It is well known that the ideal generated by the $m\times m$-minors of $D$ has height $\leq n-m+1$ (actually there is an equality if the $a_{ji}$'s are generic); since such minors are linear forms, the statement follows.

\medskip
2) By 1) we know that there are at least $h:={n \choose m}-n+m-1$ linearly independent relations among the maximal minors of $D$, hence the same holds for each $D^{(\ell)}$ (it suffices to assign the values $c_i^{(\ell)}$ to the $x_i$'s), and these relations do not depend on $\ell$.  Hence there are at least $h$ linear relations among the elements of each row of $B$, not depending on the row index $\ell$, which means that there are at least $h$ independent relations among the columns of $B$, i.e.  $rk(B)\leq n-m+1$.
\prfend

In trying to reverse the previous proposition, i.e. to deduce degeneracy of $A$ from the existence of degenerate points in the associated schemes, we have to be quite careful to avoid bi-degenerate points; namely we have:

\begin{prop}\label{A2Prop15} Let $A=(a_{ijk})$ be a tridimensional hypermatrix of format $(p,q,r)$ with $2\leq p\leq q \leq r$. If there is a degenerate point which is not a bi-degenerate point in one of the three schemes $\bL$, $\bM$ or $\bN$, then there is a point $(P,Q,T)\in K_A$,  i.e. $A$ is degenerate (so, if $r\leq p+q-1$, then $Det(A) = 0$).
\end{prop}

\proof 
Let us assume that $P=(1,0,\ldots, 0)\in \bL$ is a degenerate and not bi-degenerate point for $\bL$; we must show that there is a point $(P,Q,T)\in Ker A$. Since  $P=(1,0,\ldots ,0)$, we have $L_{jk}(P)= a_{1jk}$.

\b In the following $\Lambda $ denotes the set of all multi-indexes $\alpha= (k_1,\ldots,k_q)$, with $1\leq k_1< \ldots < k_q \leq r$, and we can number its elements: $\Lambda =\{ \alpha_1, \dots, \alpha_{r\choose q}   \}$.

Since $P\in \bL$  and it is not a bi-degenerate point, $rk\,L(P)=q-1$, so we can assume that the last $q-1$ rows of $L(P)$ are independent and that  the first row is null (modulo linear operations among the rows, which correspond to a change of coordinates in $\C^{r-1}$ which do not affect the coordinates of $P$, see Remark \ref{elop}), i.e. $$a_{11k}=0 \qquad {\rm for} \quad k=1,\dots,r. \qquad \eqno{(+)}$$
Hence the linear systems $^t\!Y\,L (P)=0 $ has a unique (projective) solution:  $Q=(1,0,\dots,0)\in \P^{p-1}$, and $Q\in \bM$ by \ref{notemp}.
Since $M=(M_{ik})=(\sum_{j=1}^q a_{ijk}y_j)$, taking $(+)$ into account we finally get
$$ 
L(P) = \left( \begin{array} {ccccc} 0&\ldots & 0\cr a_{121}&\ldots &a_{12r}\cr  \vdots  &\ldots &\vdots \cr a_{1q1}&\ldots &a_{1qr}\end{array} \right), \qquad M(Q) = \left( \begin{array} {ccccc} 0&\ldots & 0\cr a_{211}&\ldots &a_{21r}\cr  \vdots  &\ldots &\vdots \cr a_{p11}&\ldots &a_{p1r}\end{array} \right)
$$
If we prove that the system in $r$ variables $$
\left\{ \begin{array}{l}
L(P)\,Z=0 \\
\\
M(Q)\,Z=0 \end{array}\right. \qquad \eqno{(\ddagger)}
$$
 has rank $\leq r-1$, denoting by $T$ one of the (projective) solutions we finally get $^t\!Q\,L (P)= L(P)\,T=M(Q)\,T=0$, that is, $(P,Q,T) \in K_A$ (see \ref{matid}) and $A$ is degenerate, as required.
 \b If we are in the boundary or exterior case, i.e. if $r\geq p+q-1$, then there are at most $p-1+q-1$ not null equations in system $(\ddagger)$, so its rank is $\leq p+q-2\leq r-1$ and we are done (in this case we don't even need the assumption $P$ non bi-degenerate).
\b If we are in the interior case $r< p+q-1$, we proceed as follows.
Let us denote by $L_2,\dots, L_q$, respectively $M_2,\dots, M_p$ the last $q-1$, respectively $p-1$, rows in $L(P)$, respectively $M(Q)$; we want to prove that $dim \,( \langle L_2,\dots, L_q, M_2,\dots, M_p\rangle) \leq r-1$.
\b Since $P=(1,0,\ldots ,0)$ and $(+)$ holds,  we are in the conditions to apply $(\heartsuit)$ in \ref{jacobfunct}, so we have that the entries $\partial_iL_{\alpha}(P)$ of the Jacobian matrix $J_{\bL}(P)$ are 

 $$\partial_1 L_\alpha (P)=0, \qquad \partial_iL_{\alpha}(P) =  \left| \begin{array} {cccc} a_{i1k_1}&\ldots&a_{i1k_q}\cr a_{12k_1}&\ldots&a_{12k_q}\cr \vdots &\ldots &\vdots \cr a_{1qk_1}&\ldots&a_{1qk_q}  \end{array} \right| {\; \rm for} \quad i=2,\dots, p, \qquad \alpha \in \Lambda \eqno{(\clubsuit)}$$
and the first row of the $p\times {r\choose q}$-matrix $J_\bL(P)$ is made of zeros.

The point $P$ being degenerate for $\bL$, $rk\, J_\bL(P) <  \min \{r-q+1, p\}$. Since we are in the assumption $r< p+q-1$, we have $rk\,J_{\bL}(P)\leq r-q$, i.e. there are at least $p+q-r-1>0$ independent linear relations among its last $p-1$ rows:
$$ \sum_{i=2}^p \lambda_i^{(\ell)}\, \bigl(\partial_iL_{\alpha_1}(P),\dots, \partial_iL_{\alpha_{r\choose q}}(P)\bigr)=(0,\dots,0), \qquad \quad \; \ell = 1,\ldots , p+q-r-1$$
This and $(\clubsuit)$ give: 
$$
0 = \sum_{i=2}^p \lambda_i^{(\ell)}\partial_iL_{\alpha}(P) =  \left| \begin{array} {cccc} \sum_{i=2}^p \lambda_i^{(\ell)}a_{i1k_1}&\ldots &\sum_{i=2}^p \lambda_i^{(\ell)}a_{i1k_q}\cr  a_{12k_1}&\ldots & a_{12k_q}\cr \vdots &\ldots&\vdots\cr a_{1qk_1}&\ldots & a_{1qk_q} \end{array} \right|, \qquad \alpha \in \Lambda, \quad  \ell = 1,\ldots , p+q-r-1
$$
For any $\ell$ these are the $q\times q$ minors of the following $(q\times r)$-matrices $\Omega_\ell$:
$$
\Omega_\ell := \left( \begin{array} {cccc} \sum_{i=2}^p \lambda_i^{(\ell)}a_{i11}&\ldots &\sum_{i=2}^p \lambda_i^{(\ell)}a_{i1r}\cr  a_{121}&\ldots & a_{12r}\cr \vdots &\ldots&\vdots\cr a_{1q1}&\ldots & a_{1qr} \end{array} \right)
$$
hence $\Omega_\ell$ does not have maximal rank for any $\ell = 1,\ldots , p+q-r-1$. 

The last $q-1$ rows of each $\Omega_\ell$ are exactly $L_2,\dots, L_q$, which are linearly independent since $P$ is not bi-degenerate, so we are saying that for $\ell = 1,\ldots , p+q-r-1$ 
$$ \lambda_2^{(\ell)}(a_{211}, \dots,a_{21r} )+ \dots +  \lambda_p^{(\ell)}(a_{p11}, \dots,a_{p1r} ) \in \langle L_2,\dots,L_q\rangle,$$
that is, that $p+q-r-1$ linear independent combinations of $M_2,\dots,M_p$ are in $\langle L_2,\dots,L_q\rangle$. Hence
$dim (\langle L_2,\dots,L_q\rangle \cap \langle M_2,\dots,M_p\rangle) \geq p+q-r-1$, that is,
$$dim (\langle L_2,\dots,L_q, M_2,\dots,M_p\rangle) \leq (p-1+q-1)-(p+q-r-1)=r-1.$$

The role of $p$ and $q$ being symmetric, an analogous discussion shows that if there is a degenerate and non bi-degenerate point $Q\in \bM$, then $A$ is degenerate. 

\dd Now suppose there is a degenerate but not bi-degenerate point $T\in \bN$: we have to prove that there are $P\in \bL$ and $Q\in \bM$ such that $(P,Q,T)\in K_A$. This proof follows the lines of the proof for $\bL$, interior case, but we write it down for the reader's convenience.
\b Since by assumption $2\leq p\leq q\leq r$, we have $q-p+1 <r$, hence the assumption $T$ degenerate for $\bN$ means  $rk\, J_\bN(T) <  \min \{q-p+1, r\}=q-p+1$.
\b We can assume $T=(1,0,\ldots,0)$, so $N_{ij}(T)=a_{ij1}$ and we can also assume that $a_{1j1}=0$ for $j=1,\ldots,q$ since $rk N(T)= p-1$.

Again since $rk N(T)= p-1$ the system $ ^t\!X\,N(T)=0$ has exactly one (projective) solution $P=(1,0\dots,0)$; since $L=(L_{ij})=(\sum_{i=1}^p a_{ijk}x_i)$ and $a_{1j1}=0$ we finally get
$$ 
L(P) = \left( \begin{array} {ccccc} 0&a_{112}&\ldots & a_{11r}\cr   \vdots  &\vdots & &\vdots \cr 0&a_{1q2}&\ldots &a_{1qr}\end{array} \right), \qquad 
N(T) = \left( \begin{array} {ccccc} 0&\ldots & 0\cr a_{211}&\ldots &a_{2q1}\cr  \vdots  &\ldots &\vdots \cr a_{p11}&\ldots &a_{pq1}\end{array} \right)
$$

If we prove that the system in $q$ variables  $$
\left\{ \begin{array}{l}
^t\!Y\,L(P)=0 \\
\\
N(T)\,Y=0 \end{array}\right. \qquad \eqno{(\ddagger)}
$$
has rank $\leq q-1$, then it has at least a solution, say $Q$, and we finally get $^t\!P\,N(T)= \,^t\!Q\,L(P)= N(T)\,Q=0$, that is (see \ref{matid}) $(P,Q,T) \in K_A$ and $A$ is degenerate.
\b If we denote by $C_2,\dots, C_r$, respectively $N_2,\dots, N_p$ the last $q-1$ columns in $L(P)$, respectively $p-1$ rows in $M(Q)$, we want hence to prove that $dim \,( \langle C_2,\dots, C_r, N_2,\dots, N_p\rangle) \leq q-1$.
\b \b In the following $\Gamma $ denotes the set of all multi-indexes $\g= (j_1,\ldots,j_p)$, with $1\leq j_1< \ldots < j_p \leq q$, and we can number its elements: $\Gamma =\{ \g_1, \dots, \g_{q\choose p}   \}$.

Since $T=(1,0,\ldots ,0)$ and $a_{1j1}=0$ for $j=1,\ldots,q$,  we are in the conditions to apply $(\heartsuit)$ in \ref{jacobfunct}, so we have that entries $\partial_kN_{\g}(P)$ of the Jacobian matrix $J_{\bN}(T)$ are 
$$\partial_1 N_\g (T)=0, \qquad  \partial_k N_\g (T)=\left|\begin{array} {ccc} a_{1j_1k} & \ldots & a_{1j_pk}\cr  a_{2j_11} & \ldots & a_{2j_p1}\cr & \ldots &  \cr a_{pj_11}&\ldots & a_{pj_p1} \end{array} \right| {\; \rm for} \: k=2,\dots, r, \quad \g \in \Gamma$$
Hence the first row of the $r\times {q\choose p}$-matrix $ J_\bN(T)$ is made of zeros, and since $rk\, J_\bN(T) \leq q-p$ there are at least $r-1-q+p>0$ independent linear relations among its last $r-1$ rows:
$$ \sum_{k=2}^r \mu_k^{(\ell)}\, \bigl(\partial_kN_{\g_1}(T),\dots, \partial_kN_{\g_{q\choose p}}(T)\bigr)=(0,\dots,0), \qquad \quad \; \ell = 1,\ldots , r+p-q-1.$$
Hence: 
$$
0 = \sum_{k=2}^r \m_k^{(\ell)}\partial_kN_{\gamma}(T) =  \left| \begin{array} {cccc} \sum_{k=2}^r \m_k^{(\ell)}a_{1j_1k}&\ldots &\sum_{k=2}^r \m_k^{(\ell)}a_{1j_pk}\cr  a_{2j_11} &  \ldots &a_{2j_p1}\cr \vdots&  &\vdots \cr   a_{pj_11}  &\ldots &a_{pj_p1} \end{array} \right|, \qquad \g \in \Gamma, \quad  \ell = 1,\ldots , r+p-q-1.
$$
So we get that for $\ell = 1,\ldots , r+p-q-1$ the following matrices do not have maximal rank:
$$
\Theta_\ell = \left( \begin{array}{cccc}\sum_{k=2}^r \m_k^{(\ell)}a_{11k}  & \ldots & \sum_{k=2}^r \m_k^{(\ell)}a_{1qk}\cr a_{211}  &\ldots &a_{2q1}\cr \vdots& &\vdots \cr   a_{p11}  &\ldots &a_{pq1}\end{array} \right).
$$
The last $p-1$ rows of each $\Theta_\ell$ are exactly $N_2,\dots, N_p$, which are linearly independent since $T$ is not bi-degenerate, so we are saying that for $\ell = 1,\ldots , r+p-q-1$ 
$$ \m_2^{(\ell)}(a_{112}, \dots,a_{1q2} )+ \dots +  \m_r^{(\ell)}(a_{11r}, \dots,a_{1qr} ) \in \langle N_2,\dots,N_p\rangle,$$
that is, that $r+p-q-1$ linear independent combinations of $C_2,\dots,C_r$ are in $\langle L_2,\dots,L_q\rangle$. Hence
$$dim (\langle C_2,\dots,C_r, N_2,\dots,N_p\rangle) \leq q-1.$$
\prfend 

\begin{remark} \rm Notice that of course there is no reason why a degenerate hypermatrix's kernel should be made only by degenerate but non-bidegenerate points. For example one (or even all) of its associated schemes could possess only bi-degenerate points; consider a $(3,3,3)-$hypermatrix where $a_{111}\neq 0$ and all the other entries $a_{ijk}=0$ (hence with $Trk\, A = 1$):  here $\bL \cong \bM \cong \bN \cong \P^2$ and all their points are bi-degenerate.
\end{remark}

We can summarize the results seen in the previous two propositions and in their proofs in the following:

\begin{thm}\label{Theorem}  Let $A=(a_{ijk})$ be a tridimensional hypermatrix of format $(p,q,r)$ with $2\leq p\leq q \leq r$, then:
\begin{itemize}
\item If  there exists $ (P,Q,T) \in K_A$, then $P, Q, T$ are degenerate points for $\bL, \bM, \bN$, respectively.
\item If $\bN$ possesses a degenerate but non bi-degenerate point, then $A$ is degenerate.
\item If $A$ is not interior (i.e. $r \geq p+q-1$), then $A$ is degenerate $\iff$ $\bL$ is degenerate $\iff $ $\bM$ is degenerate.
\item If $r < p+q-1$ and there is a degenerate non bi-degenerate point in $\bL$ or in  $\bM$ then $A$ is degenerate.

\end{itemize}
\end{thm}

\begin{remark}\label{2.23} If $r \geq p+q-1$, to say that $\bL$ is degenerate means that $\bL \neq \emptyset$ (see beginning of the proof of \ref{A2P14}). Hence Theorem \ref{Theorem} says in particular that, when $r \geq p+q-1$, $A$ is degenerate iff $\bL \neq \emptyset$ (and the same is true for $\bM$; see also \cite{O} 6.1, \cite{GKZ} XIV.3.1).

\b Let $A$ be a $(p,q,r)$ hypermatrix with $F_{ess}(A) = (p,q,r')$ and $p+q \leq r' < r$; let $A'$ be a hypermatrix with $r-r'$ 0-slices in the $z$ direction, obtained from $A$ via a linear change in the $z$-coordinates. Then, $A'$ is non-degenerate iff $A$ is non-degenerate. In fact, denoting by $\bL',\bM', \bN'$ the schemes associated to $A'$, we have that $\bL$ is projectively equivalent to $\bL'$, hence $\bL \neq \emptyset$ iff $\bL' \neq \emptyset$.
\end{remark}

\begin{exa}\label{334}  We want to check that asking that there is a degenerate and non bidegenerate point in $\bN$ is an essential requirement in Theorem \ref{Theorem}, not only when there are bi-degenerate points even for generic tensors, as in the Example \ref{335}. Let us consider a hypermatrix $A$, with $(p,q,r) = (3,3,4)$, given by the following $z$-slices:

$$A_{ij1}= \left( \begin{array} {ccc} 1&1&1\cr 0&1&1\cr 1&0&2 \end{array}\right) \quad A_{ij2}= \left( \begin{array} {ccc} 0&1&2\cr 1&0&1\cr 0&1&1 \end{array}\right)  \quad 
A_{ij3}= \left( \begin{array} {ccc} 1&1&2\cr 0&-1&0\cr 0&0&0 \end{array}\right) \quad A_{ij4}= \left( \begin{array} {ccc} 1&2&1\cr 1&1&1\cr 1&1&1 \end{array}\right)  $$

For this $A$, the scheme $\bN$ is defined by the determinant of a $3\times 3$ matrix $N$ of linear forms in $\C[z_1,z_2,z_3,z_4]$, namely:
$$
N = \left( \begin{array} {ccc} z_1+z_3+z_4&z_1+z_2+z_3+2z_4&z_1+2z_2+2z_3+z_4\cr z_2+z_4&z_1-z_3+z_4&z_1+z_2+z_4\cr z_1+z_4&z_2+z_4&2z_1+z_2+z_4 \end{array}\right)   $$
A simple computation (we used \cite{CoCoA}  for computations in Examples 3.10 and 3.11) shows that 

det $N = 2z_1^3 - 2z_1^2z_2 - 2z_1z_2^2 + z_2^3 - z_2z_3^2 + 3z_1^2z_4 - 6z_1z_2z_4 + z_2^2z_4 - 3z_1z_3z_4 + 3z_2z_3z_4 + z_3^2z_4 - z_2z_4^2$,

hence $\bN$ is the cubic surface in $\P^3$ defined by such equation. By computing the four partial derivatives, we get that the Jacobian ideal of $\bN$ is 

$J_\bN = (6z_1^2 - 4z_1z_2 - 2z_2^2 + 6z_1z_4 - 6z_2z_4 - 3z_3z_4,-2z_1^2 - 4z_1z_2 + 3z_2^2 - z_3^2 - 6z_1z_4 + 2z_2z_4 + 3z_3z_4 - z_4^2,-2z_2z_3 - 3z_1z_4 + 3z_2z_4 + 2z_3z_4,3z_1^2 - 6z_1z_2 + z_2^2 - 3z_1z_3 + 3z_2z_3 + z_3^2 - 2z_2z_4)$,

and the radical ideal of $J_\bN$ is $I = (z_3, z_2 + z_4, z_1 + z_4)$, hence $\bN$ possesses one (and only one) singular point $P=(1,1,0,-1)$, which is degenerate for $N$. The matrix $N(P)$ is 
$$
\left( \begin{array} {ccc} 0&0&2\cr 0&0&1\cr 0&0&2 \end{array}\right)   $$
whose rank is $=1$, hence $P$ is a bi-degenerate point, and we cannot conclude anything about $A$. Let us show that actually $A$ is not degenerate.

The matrices $L$ and $M$ are 
$$
L = \left( \begin{array} {cccc} x_1+x_2+x_3&x_2+2x_3&x_1+x_2+2x_3&x_1+2x_2+x_3\cr x_2+x_3&x_1+x_3&-x_2&x_1+x_2+x_3\cr x_1+2x_3&x_2+x_3&0&x_1+x_2+x_3 \end{array}\right) ,  $$
$$
M = \left( \begin{array} {cccc} y_1+y_3&y_2&y_1&y_1+y_2+y_3\cr y_1+y_2&y_1+y_3&y_1-y_2&2y_1+y_2+y_3\cr y_1+y_2+y_3&2y_1+y_2+y_3&2y_2&y_1+y_2+y_3 \end{array}\right).  $$

The ideal of $\bL$, given by the $3\times 3$ minors of $L$ is:
$
I_\bL = (-x_1^3 - x_1^2x_2 + x_1x_2^2 + 2x_2^3 - 5x_1^2x_3 - 2x_1x_2x_3 + 4x_2^2x_3 - 7x_1x_3^2 - 2x_3^3, -2x_1x_2x_3 + x_2^2x_3 + x_1x_3^2 + x_3^3, x_1^3 + x_1^2x_2 - x_1x_2^2 - 2x_2^3 + 4x_1^2x_3 + 3x_1x_2x_3 + 5x_1x_3^2 + 2x_2x_3^2 + 2x_3^3, -x_1^3 - x_1^2x_2 + x_1x_2^2 + 2x_2^3 - 3x_1^2x_3 - x_1x_2x_3 + 3x_2^2x_3 - 2x_1x_3^2 + x_2x_3^2).
$
This ideal is a radical one (by \cite{CoCoA}), and the Hilbert function of $R/I_\bL$ is $H(0) = 1, H(1) = 3, H(t) = 6   , t \geq 2$, i.e. $\bL$ is made of six distinct point, with no degenerate ones, hence, by Theorem 3.9, $A$ is not degenerate.

When considering the geometry of the situation, what happens here is that there are three points on a line in $\bL$: the ideal $J=(z)+I_\bL$ is such that the  Hilbert function of $R/J$ is: $H(0) = 1, H(1) = 2, H(t) = 3 $  for $t \geq 2 $. So the peculiarity of $\bL$ is not to have degenerate points, but three of them on a line. Since, as it is classically known, $\bN$ is the surface obtained via the linear system of the cubics in $\P^2$ passing through $\bL$, the line containing three points of $\bL$ get contracted to a point ($P$) which is singular on $\bN$ (the same happens for $\bM$).  In the generic case, with no three points on a line, $\bN$ is smooth and the double-six of lines on it is given by the exceptional lines of the two blow-ups of $\P^2$ on $\bL$ and on $\bM$.

\end{exa}

Let us consider now another example, where we can use Theorem \ref{Theorem}  to determine that a (3,3,4) tensor $A$ is degenerate:

\begin{exa} Let $A$ be given by the following $z$-slices:

		$$A_{ij1}= \left( \begin{array} {ccc} 1&0&0\cr 0&0&1\cr 1&0&0 \end{array}\right) \quad A_{ij2}= \left( \begin{array} {ccc} 0&0&1\cr 0&1&0\cr 0&0&0 \end{array}\right)  \quad 
A_{ij3}= \left( \begin{array} {ccc} 0&0&0\cr 0&0&0\cr 0&0&1 \end{array}\right) \quad A_{ij4}= \left( \begin{array} {ccc} 0&1&0\cr 1&0&0\cr 0&0&0\end{array}\right)  $$
		
The scheme $\bN$ is defined by the determinant of $N$: 
$$
N = \left( \begin{array} {ccc} z_1&z_4&z_2\cr z_4&z_2&z_1\cr z_1&0&z_3 \end{array}\right)   $$
Which gives:   det$(N) = -z_1z_2^2 + z_1z_2z_3 + z_1^2z_4 - z_3z_4^2$.    

 The  Jacobian Ideal of $\bN$ is  $J_\bN = (-z_2^2 + z_2z_3 + 2z_1z_4,-2z_1z_2 + z_1z_3,z_1z_2 - z_4^2,z_1^2 - 2z_3z_4)$ and its Hilbert function is: $H(0) = 1, H(1) = 4, H(2) = 6, H(3) = 4, H(t) = 2, t \geq 4$, while its radical ideal is $(-z_2^2 + z_2z_3, z_4, z_1)$ . Hence $\bN$ has two singular points: $(0,0,1,0)$ which is bidegenerate for $\bN$ and $(0,1,1,0)$ which is degenerate but not bidegenerate, hence $A$ is degenerate.     

The matrices $L$ and $M$ are 
$$
L = \left( \begin{array} {cccc} x_1&x_3&0&x_2\cr x_3&x_2&0&x_1\cr x_1&0&x_3&0 \end{array}\right) ,  $$
$$
M = \left( \begin{array} {cccc} y_1+y_3&0&0&y_2\cr 0&y_2&0&2y_1\cr y_2&y_1&y_3&0 \end{array}\right).  $$

The ideal $I_\bL$ is $(x_1x_2x_3 - x_3^3, -x_1x_2^2 + x_1^2x_3, -x_1^2x_3 + x_2x_3^2, x_2^2x_3 - x_1x_3^2])$, and the Hilbert function of $(R/I_\bL)$ is $H(0) = 1,H(1) = 3,H(t) = 6,, t\geq 2$, while its radical is $(x_1x_2 - x_3^2, x_2^2x_3 - x_1x_3^2, x_1^2x_3 - x_2x_3^2)$, whose Hilbert function is $H(0) = 1, H(1) = 3, H(t) = 5 , t \geq 2$, so $\bL$ is made of five points, with $(0,1,0)$ which is double and the others are simple; let us notice that $(0,1,0)$ is bidegenerate for $\bL$.

An analogous analysis for $\bM$ shows that it is made by five points too, with $(1,0,-1)$ double, but non bidegenerate this time.

\end{exa}

\section{Tensors of format $(2, 2, r)$.} 

This section is dedicated to the study of hypermatrices with this kind of format, whose invariants can be more easily known; our aim is to show how the analysis of the associated schemes allows to write algorithms which compute at the same time the tensor rank, the essential form and the degeneracy. 

The technique implied is the extended approach ``a la  Sh\"{a}fli" developed in the previous section, which implies to consider all the three associated schemes $\bL$, $\bM$, $\bN$; moreover we use the canonical forms for these hypermatrices in case $2\times 2 \times 2$ (see \cite{E}), from which canonical forms for the cases $2\times 2 \times r$ can easily be deduced.

\subsection{Tensors of format $(2, 2, 2)$} 
\b We consider here the simplest 3-dimensional case, i.e. the $(2,2,2)$ format (see also \cite{CGG1}, \cite{A1}, \cite{O}). We are in the interior case, hence the hypermatrix A is degenerate iff Det A=0. Recall that by theorem 3.8 we have that A is degenerate iff L is degenerate iff M is degenerate; moreover, N degenerate implies A degenerate. 
\b Since we always consider tensors and hypermatrices modulo moltiplication for a non-zero constant, we can view $A$ as an element of the space $\P^{7}$ parameterizing trihomogeneous polynomials in $\C[x_1,x_2,y_1,y_2,z_1,z_2]_{1,1,1}$, and the Segre Variety $V_{1,1,1}\subset \P^{7}$ as the one parameterizing decomposable tensors, i.e. tensors of rank 1. In this case we know by \cite{E} that with linear changes of coordinates  $A$ can be reduced to one of the following cases ({\it canonical forms}): 

\begin{figure}[H]
		%\centering
		\includegraphics[scale=0.5]{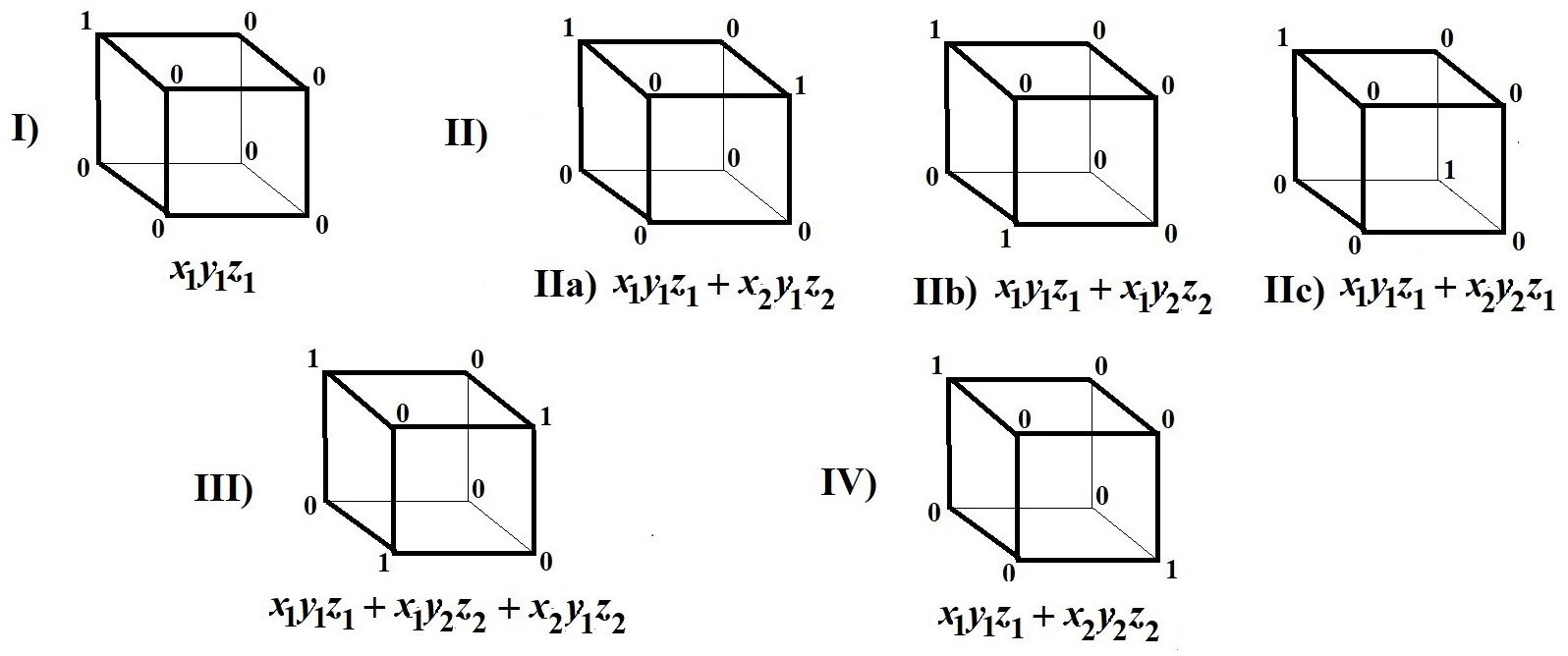}
			\end{figure}

When $A$ is of type I or II $A$ is not concise (hence $Det(A) = 0$), while in cases III and IV, we have concise hypermatrices. In case III again $Det(A)=0$, see \cite{CGG1}, while type IV is the generic one, with $Det(A)\neq 0$. Notice that for this format we have an explicit formula for the hyperdeterminant (see \cite{GKZ}); if $A = (a_{ijk})$, then:

$Det(A) = (a^2_{111}a^2_{222} + a^2_{112}a^2_{221} + a^2_{121}a^2_{212} + a^2_{122}a^2_{211})
- 2(a_{111}a_{112}a_{221}a_{222} + a_{111}a_{121}a_{212}a_{222} + a_{111}a_{122}a_{211}a_{222}+
a_{112}a_{121}a_{212}a_{221} + a_{112}a_{122}a_{221}a_{211} + a_{121}a_{122}a_{212}a_{211})+
4(a_{111}a_{122}a_{212}a_{221} + a_{112}a_{121}a_{211}a_{222}).$

In this case, the hyperdeterminant gives exactly the tangential variety of $V_{1,1,1}$: $\mathbb{V} = \tau(V_{1,1,1})$. Actually, since $\dim \tau(V_{1,1,1})= 2\dim V_{1,1,1} = 6$,  $\tau(V_{1,1,1})$ is a hypersurface. Moreover, we have that all the three schemes $\bL$, $\bM$ and $\bN$ are given by a quadratic equation in $\P^1$, these three quadrics have the same discriminant and the hyperdeterminant coincides with this discriminant (see \cite{O}).

When we look at the tensor rank of $A$, $Trk\, A$, we have that the generic tensor rank is 2, and the hypersurface $\mathbb{V}$ is the closure of the locus of the tensors with $non-generic$ rank (as remarked in \cite{CGG1}): the generic degenerate hypermatrix has $Trk\, A=3$ (which is the maximum value for $Trk\, A$, differently from what happens with ordinary 2-dimensional matrices, where generic and maximum rank coincide). Here the degenerate hypermatrices are either the ones with maximum tensor rank: $Trk\, A =3$ (third case), or the ones corresponding to decomposable tensors (first case: $A$ is corresponds to a point of $V_{1,1,1}$, which is the locus of  tensors with $Trk\, A= 1$), or the ones of type II), which are the ones corresponding to non-concise tensors with $Trk\, A=2$, for which $F_{ess}(A)$ is either $(1,2,2)$, $(2,1,2)$ or $(2,2,1)$.

In the generic case IV, by Theorem \ref{simpleformats}, we have that the three schemes $\bL$, $\bM$ and $\bN$ are all non-degenerate.  When we are in one of the degenerate cases, all the three schemes will be degenerate. More specifically (we will use coordinates $w_1,w_2$ when we show which kind of matrix we get, without specifying for which of the $x_i$'s, $y_j$'s or $z_k$'s):
\begin{itemize}
\item In case I the three matrices are $\left( \begin{array} {cc}  x_1&0\cr 0&0 \end{array}\right)$, $\left( \begin{array} {cc}  y_1&0\cr 0&0 \end{array}\right)$ and $\left( \begin{array} {cc}  z_1&0\cr 0&0 \end{array}\right)$, respectively, and all the three schemes coincide with $\P^1$.
\item In cases II two of the matrices  $L$, $M$, $N$ are of type
$\left( \begin{array} {cc}  w_1&0\cr w_2&0 \end{array}\right)$ or its transpose, and the third  of type $\left( \begin{array} {cc}  w_1&0\cr 0&w_1 \end{array}\right)$.  This means that two among the schemes $\bL$, $\bM$ and $\bN$ coincide with the whole $\P^1$, while the other is a double point, hence singular (that is degenerate) for the scheme and also bi-degenerate.

\item In case III the three matrices $L, M$ and $N$ are of type $\left( \begin{array} {cc}  w_1&w_2\cr 0&w_1 \end{array}\right)$, and all the three schemes are given by a double point.

\item in case IV the three matrices are of type $\left( \begin{array} {cc}  w_1&0\cr 0&w_2 \end{array}\right)$, hence the three schemes $\bL$, $\bM$ and $\bN$ are made of two simple points ($(1,0)$ and $(0,1)$) in $\P^1$.

\end{itemize}

\medskip
Since the behaviours described above are invariant for linear changes of coordinates, what we saw allows us to give a simple algorithm to determine the tensor rank, the conciseness and the degeneracy of a hypermatrix $A$ of format $(2,2,2)$. 

The problem reduces to the study of three binary quadratic forms, which is computationally quite easy to deal with. We use the same notations as above.

\begin{algorithm}[H]\caption{Tensor rank, hyperdeterminant and conciseness for $(2,2,2)$ tensors} \label{Alg2x2x2}
\a  \textbf{Input}: $A = (a_{ijk})$, $i,j,k \in \{1,2\}$ $a_{ijk} \in \C$.\\
 \a \textbf{Output}: State if $Det(A)=0$, if $A$ is concise or not, and compute $Trk$ and essential format of $A$ .

\a \begin{boxedminipage}{135mm}
    \begin{algorithmic}[1]

	\STATE\label{Alg1Step1} {\it STEP 1)} \  Compute $\det L\in \C[x_1,x_2]$:
	
	- If it has 2 distinct zeros, then $Det(A) \neq 0$, $A$ is concise and $Trk\, A=2$. {\bf STOP}
	
	- If $\det L\in \C[x_1,x_2]$ has a double root, or is 0, then $Det(A) =0$ and the schemes $\bL$, $\bM$ and $\bN$ are degenerate; go to the next step.
	
	   \STATE\label{Alg1Step2}{\it STEP 2)} \  Compute $\det M$ and $\det N$; then:	 
	   
	   - If all the three computed determinants have a double root, then $A$ is concise and $Trk\, A=3$. {\bf STOP}
	  
	 - If two among $\det L$, $\det M$, $\det N$,  vanish and the other has a double root, then $A$ is not concise, $Trk\, A=2$, and $F_{ess}(A)$ is either $(1,2,2)$, $(2,1,2)$ or $(2,2,1)$, depending on which among $\bL$, $\bM$ or $\bN$ is a double point. {\bf STOP}
	  
 	  - If the three computed determinants are all 0, then $Trk\, A= 1$, $A$ is non-concise and $F_{ess}(A)=(1,1,1)$. {\bf STOP}
	  
  \end{algorithmic}
\end{boxedminipage}
\end{algorithm}

\subsection{Tensors of format $(2, 2, 3)$}

\b This is what we called a boundary format, for which the hyperdeterminant is still defined.
A hypermatrix of this format appears as follows:

\begin{figure}[H]
		%\centering
		\includegraphics[scale=0.35]{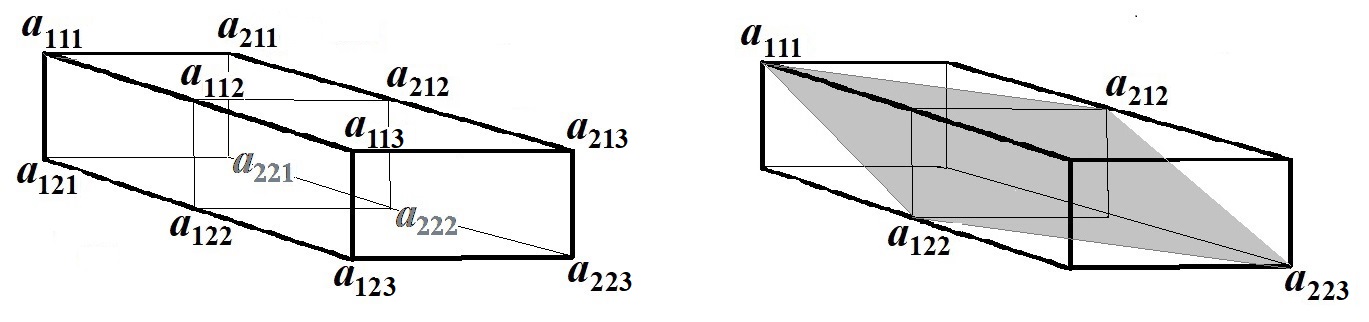}
			\end{figure}
			
Actually, it is not hard to check that with linear operations on its slices, $A$ can be reduced to what it is considered a {\it diagonal format}, as in the picture above: what is considered as the diagonal is evidentiated in grey (e.g. see \cite{O}). When $A$ is in diagonal form, we have $Det(A) = a_{111}^2a_{212}a_{122}a_{223}^2$ (\cite{O} 7.8).

All this leads us to consider the following cases, which are ``canonical forms", in the sense that every hypermatrix can be reduced to one of these via coordinate changes in the three spaces:

\begin{figure}[H]
		%\centering
		\includegraphics[scale=0.35]{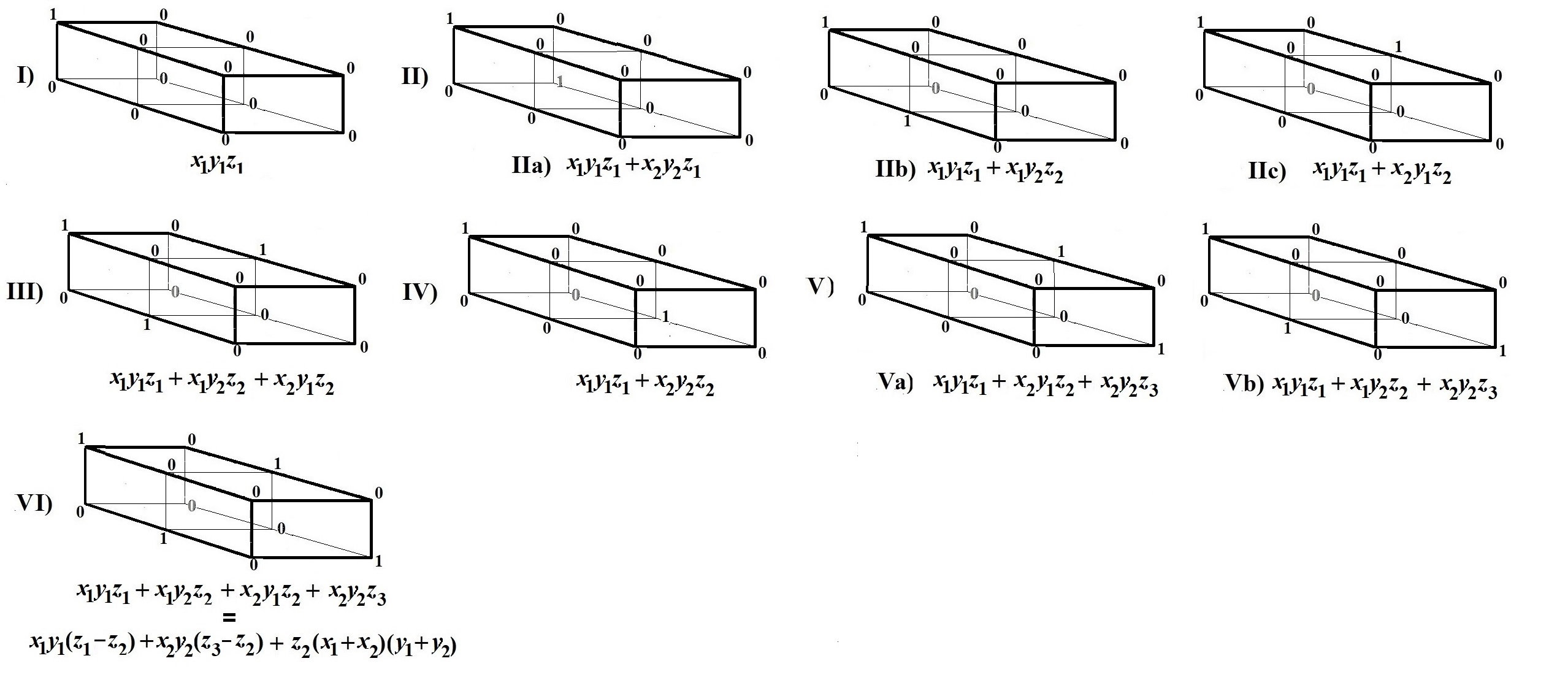}
			\end{figure} 
			
Cases I to IV are the non-concise ones and correspond to the cases we have already seen in the $(2,2,2)$ case. It is easy to check, using the cases (2,2,2) and linear operations among slices, that concise cases can be reduced to cases V or to case VI. Cases V are concise because no slice combination can produce a 0-slice; moreover, $Det(A) = 0$. The last case VI is the generic one (it is the diagonal format), non-degenerate and concise. The cases III, Va, Vb and VI represent tensors with $Trk\, A=3$, since if their rank were $\leq 2$, we shoud be in one of the cases I, II or IV, because with a linear change of coordinates only two monomials would appear. 

 Let us analyze the schemes $\bL$, $\bM$ and $\bN$ more in detail ($w_1$ and $w_2$ are used as in the case $(2,2,2)$):
\begin{itemize}
\item In case I we have that $N = \left( \begin{array} {cc}  z_1&0\cr 0&0 \end{array}\right)$, while $L$ and $M$ are of type $\left( \begin{array} {ccc}  w_1&0&0\cr 0&0&0 \end{array}\right)$, hence $\bN = \P^2$, while $\bL=\bM=\P^1$.
\item In case IIa the matrices  $L$ and $M$ will look like:
$\left( \begin{array} {ccc}  w_1&0&0\cr w_2&0&0\end{array}\right)$, while $N = \left( \begin{array} {cc}  z_1&0\cr 0&z_1 \end{array}\right)$.  This means that  $\bL=\bM=\P^1$ while $\bN$ is a double line in $ \P^2$. 
\item In cases IIb, IIc the matrix $N$ is either $\left( \begin{array} {cc}  z_1&0\cr z_2&0 \end{array}\right)$ or its transpose, while $L$ and $M$ are: $\left( \begin{array} {ccc}  x_1&x_2&0\cr 0&0&0 \end{array}\right)$, $\left( \begin{array} {ccc}  y_1&0&0\cr 0&y_1&0 \end{array}\right)$, hence  $\bN=\P^2$, while one between $\bL$ and $\bM$ is the whole $\P^1$ and the other is a double point.
\item In case III we have:  $L$ and $M$ like $\left( \begin{array} {ccc}  w_1&w_2&0\cr 0&w_1&0 \end{array}\right)$, while $N = \left( \begin{array} {cc}  z_1&z_2\cr z_2&0 \end{array}\right)$, hence $\bL$ and $\bM$ are given  both by a double point in $\P^1$, while $\bN$ is a double line in $\P^2$.
\item in case IV we have:  $L$ and $M$ are $\left( \begin{array} {ccc}  w_1&0&0\cr 0&w_2&0 \end{array}\right)$, while $N = \left( \begin{array} {cc}  z_1&0\cr 0&z_2 \end{array}\right)$; $\bL$ and $\bM$ are two simple points ($(1,0)\cup (0,1) \in \P^1$), while $\bN$ is made of two lines ($\{z_1=0\}\cup \{z_2=0\}$) in $\P^2$.
\item in case Va we have:  $L = \left( \begin{array} {ccc}  x_1&0&0\cr 0&x_1&x_2 \end{array}\right)$,  $M =\left( \begin{array} {ccc}  y_1&y_2&0\cr 0&0&y_2 \end{array}\right)$, $N = \left( \begin{array} {cc}  z_1&0\cr z_2&z_3 \end{array}\right)$; $\bL$ and $\bM$ are given by the ideals $(x_1^2,x_1x_2)$ and $(y_1y_2,y_2^2)$ and both are a simple point in $\P^1$ (here the determinantal ideal is not saturated, it is the degree 2 part of the ideal of the point); $\bN$ is made of two lines ($\{z_1=0\}\cup \{z_3=0\}$) in $\P^2$. The case Vb is analogous.
\item in case VI  $L$ and $M$ are $ \left( \begin{array} {ccc}  w_1&w_2&0\cr 0&w_1&w_2 \end{array}\right)$, $N = \left( \begin{array} {cc}  z_1&z_2\cr z_2&z_3 \end{array}\right)$; $\bL = \bM = \emptyset$, given by the irrelevant ideal $(w_1^2,w_1w_2,w_2^2)$, while $\bN$ is defined by $z_1z_3-z_2^2$, a smooth conic in $\P^2$.
\end{itemize} 

Recalling (see Definition \ref{determ}) that the expected codimension of $\bL$ and $\bM$ is 2, i.e. they are expected to be empty, while the expected codimension for $\bN$ is 1, and applying theorem \ref{Theorem}, we see that A is non degenerate only in case VI). Now we can give an algorithm for hypermatrices of this format:

\begin{algorithm}[H]\caption{Tensor rank, hyperdeterminant and essential format for $(2,2,3)$ tensors} \label{Alg2x2x3}
\a  \textbf{Input}: $A = (a_{ijk})$, $i,j \in \{1,2\}$, $k\in \{1,2,3\}$, $a_{ijk} \in \C$.\\
 \a \textbf{Output}: State if $Det(A)=0$, determine $F_{ess}(A)$, and compute $Trk\, A$.

\a \begin{boxedminipage}{135mm}
    \begin{algorithmic}[1]

	\STATE\label{Alg2Step1} {\it STEP 1)} \  Compute $\det N\in \C[z_1,z_2,z_3]$:
	
	- If it defines a smooth conic, then $Det(A) \neq 0$, $A$ is concise and $Trk\, A=3$. {\bf STOP}
	
	- If $\det N$ defines a singular conic or is 0, then the schemes $\bL$, $\bM$ and $\bN$ are degenerate and $Det(A) =0$; GO TO NEXT STEP.
	
	   \STATE\label{Alg2Step2}{\it STEP 2)} \  If $\det N$ defines a conic made by two distinct lines, compute the minors of $L$; then:	 
	   
	   - If the minors of $L$ define one point in $\P^1$, then $A$ is concise and $Trk\, A=3$. {\bf STOP}
		
		 - If the minors of $L$ define two points in $\P^1$, then $A$ is non-concise, $F_{ess} = (2,2,2)$ and $Trk\, A=2$. {\bf STOP}
	  
\medskip
	 If $\det N$ does not define a conic made by two distinct lines, $A$ is not concise. GO TO NEXT STEP.
	
	\STATE\label{Alg2Step3}{\it STEP 3)} \ If $\det N$ defines a conic made by a double line, compute the minors of $L$; then:
	
	- If the minors of $L$ define a double point in $\P^1$, then $F_{ess}(A) = (2,2,2)$ and $Trk\, A=3$. {\bf STOP}
		
		 - If the minors of $L$ vanish identically, i.e. $\bL=\P^1$, then  $F_{ess} = (2,2,1)$ and $Trk\, A=2$. {\bf STOP}

\medskip 
	 If $\det N$ does not define a conic made by a double line, GO TO NEXT STEP.
	
	 	 		\STATE\label{Alg2Step4}{\it STEP 4)} \ When $\det N$ vanishes identically, i.e. $\bN=\P^2$, compute the minors of $M$ and of $L$; then:	 
	   
	   - If one between $\bL$ and $\bM$ is a double point, then $F_{ess} = (1,2,2)$ or $F_{ess} = (2,1,2)$ and $Trk\, A=2$. {\bf STOP}
	  
	 - If $\bL = \bM=\P^1$, i.e. all the minors of the two matrices are identically 0, then $Trk\, A= 1$  and $F_{ess}(A)=(1,1,1)$. {\bf STOP}
	  
  \end{algorithmic}
\end{boxedminipage}
\end{algorithm}

\subsection{Tensors of format $(2, 2, r)$, $r\geq 4$}

In this section we use the relation between tensor rank and secant varieties $\sigma_k(X)$ of the Segre variety $X$; for definitions and more details see e.g. \cite{BCCGO}.

For $(2,2,r)$ tensors with $r\geq 4$ the hyperdeterminant is not defined. With regard to tensor rank, in these cases we can apply the following result (see \cite{CGG1},\cite{CGG2}):

\begin{thm}\label{unbalanced} Let $X$ be the Segre embedding of $\P^{n_1} \times \ldots \times \P^{n_t} \times \P^n \rightarrow \P^m$, where \linebreak $m = (n+1)\prod_{i=1}^t(n_i+1)-1$. Let $h =  \prod_{i=1}^t(n_i+1)-1$ and assume $n > h-\sum_{i=1}^tn_i+1$. Then, the $s$-secant variety $\sigma_s(X)$ is defective for
$$ h-\sum_{i=1}^tn_i+1 < s \leq \min\{n,h\},$$
with defect equal to $\delta_s(X) = s^2-s(h- \sum_{i=1}^tn_i+1)$.
\end{thm}

In our case, for format $(2,2,r)$, we have that $m=4r-1$ and $h=3$, so we get that $\sigma_3(V_{1,1,r-1})$ is defective, with $\delta_3(V_{1,1,r-1})=3$. This implies that for $r=4,5,6$, we should have that $\sigma_3(V_{1,1,r-1})$ fills up the ambient space $\P^m$, but it doesn't, for its defectivity. Recall that the generic point of $\sigma_s (V_{1,1,r-1})$ represents the tensor of rank $s$. Since $\sigma_4 (V_{1,1,r-1})= \P^m$,   the generic rank for a tensor of  format $(2,2,4)$, $(2,2,5)$ and $(2,2,6)$ is 4, as it happens for $r\geq 7$ when 4 is the expected generic rank. 

We now examine the hypermatrices of format $(2,2,r)$, $r\geq 4$, more in detail, and this will allow us to prove Proposition \ref{2,2,4} below and to give an algorithm for tensor rank, degeneracy and essential format.

\begin{prop}\label{2,2,4} Let $A$ be a generic hypermatrix of format $(2,2,r)$, with $r\geq 4$; then $F_{ess}(A) = (2,2,4)$, hence  $A$ is non-concise if and only if $r\geq 5$ .   Moreover, $Trk\, A=4$ and this is also the maximum rank for hypermatrices of this format. 
\end{prop} 

Recall that (see Definition \ref{determ}) the expected codimension of the associated schemes $\bL$ and $\bM$ is 2, i.e. they are expected to be empty, while the expected codimension for $\bN$ is 1. Moreover, theorem \ref{Theorem} and \ref{2.23} tell us that $A$ is degenerate iff $\bL \neq \emptyset$, iff $\bM  \neq \emptyset$, and that if $\bN$ has a degenerate non-bidegenerate point then $A$ is degenerate. We shall also need the following

\begin{remark}\label{AA'}  Let $A$ be a hypermatrix $(2,2,r)$, let $A'$ be the sub-hypermatrix formed by the first $s$ slices in the $z$-direction, and assume that the other $z$-slices are 0-slices. Denoting by $\bL', \bM', \bN'$ the schemes associated to $A'$, we have that $\bL=\bL'$ and $\bM=\bM'$, while $\bN$ is a cone over $\bN' \subset \P^{s-1}$ with vertex $z_{s+1}=\dots =z_r=0$ in $\P^{r-1}$.
\end{remark}

\b $i)$ Let $A$ be concise of format $(2,2,4)$; then, it can be reduced to the form described in the following figure by linear operations on the slices:

\begin{figure}[H]
		%\centering
		\includegraphics[scale=0.45]{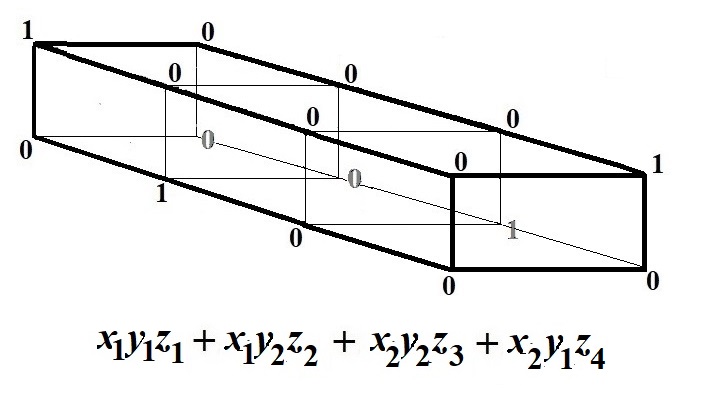}
			\end{figure}
			
Actually, if $A$ is concise, the sub-hypermatrix formed by its first 3 slides in the $z$-direction must be concise, hence can be reduced as in cases V or VI for format (2,2,3). If we are in case Va we can use the first three $z$-slices to eliminate the entries on the fourth $z$-slice  corresponding to the entries equal to 1, so to get the form above. In case Vb, we have an analogous situation. In case VI we can use the first three $z$-slice to eliminate three of the entries on the fourth, ending up with a form (with $a\neq 0$, $A$ being concise):

\begin{figure}[H]
		%\centering
		\includegraphics[scale=0.45]{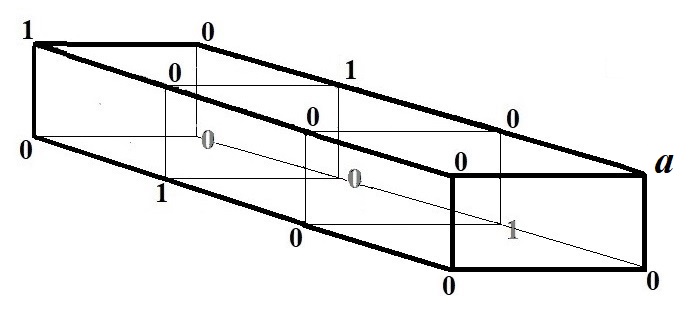}
			\end{figure}	
			
Now we can use the fourth $z$-slice to eliminate one entry on the second, to end up as desired.

Notice that $Trk A = 4$, in fact it is $\leq 4$ from above, and if it were to be $\leq 3$ then $P_A$ could be written as a sum of three monomials where at most three among the variables $z_1,z_2,z_3,z_4$ appear, hence $A$ would not be concise. 
\b The associated scheme $\bL$ is the determinantal scheme given by  $L= \left( \begin{array} {cccc} x_1&0&0&x_2\cr 0&x_1&x_2&0 \cr \end{array}\right)$, that is, $\bL $ is empty since $I_{\bL}=(x_1^2,x_1x_2,x_2^2)$ in $\P^1$, so we conclude by \ref{2.23} that $A$ is non-degenerate.
\b The associated scheme $\bN$ is the determinantal scheme given by  $N= \left( \begin{array} {cc} z_1&z_4\cr z_2&z_3 \cr \end{array}\right)$, that is, $\bN $ is a smooth quadric  in $\P^3$.

\a  $ii)$ Let $A$ be a hypermatrix of format (2,2,r) with $r\geq 5$; then it is non-concise by \ref{concisealg}. If it has maximal essential format (2,2,4), it is clear that it can be reduced to a form where the first four $z$-slices are as in case $i)$, and the following ones are all 0-slices.
 \b Since the ideal of the scheme $\bL$ is the same as in case $i)$, such a hypermatrix is non-degenerate.
\b The scheme $\bN$ is again given by  $N= \left( \begin{array} {cc} z_1&z_4\cr z_2&z_3 \cr \end{array}\right)$, that is, $\bN $ is a cone in $\P^{r-1}$ over a smooth quadric of $\P^3$,  and its vertex is a linear space of dimension $r-5$ in $\P^{r-1}$; notice that all the singular points of $\bN$ are bi-degenerate points, so the singularity of $\bN$ does not imply the degeneracy.

\a $iii)$ Let $A$ be non-concise of format $(2,2,4)$; then its $x$-slices, or $y$-slices, are a matrix $2 \times r$ and a 0-slice, in which case its tensor rank is $\leq 2$, or it can be written as a hypermatrix of format (2,2,3) with a  $z-$slice of zeroes, hence its tensor rank is $\leq 3$.		
\b In the second case Algorithm 2 gives the analysis for essential format and degeneracy, keeping in mind \ref{AA'}. 

\a $iv)$ Let $A$ be non-concise of format $(2,2,r)$, $r\geq 5$; then it reduces to a matrix $2 \times r$, in which case its tensor rank is $\leq 2$, or it can be written as a hypermatrix of format $(2,2,4)$ with some  $z$-slice of zeroes, hence its tensor rank is $\leq 4$.	
\b In the second case the properties of $A$ are deduced by the previous analysis for the $(2,2,4)$ format.

\begin{algorithm}[H]\caption{Tensor rank, degeneracy and essential format for $(2,2,r)$ tensors, $r\geq 4$} \label{Alg2x2xr}
\a  \textbf{Input}: $r$ a natural number $\geq 4$, $A = (a_{ijk})$, $i,j \in \{1,2\}$, $k\in \{1,...,r\}$, $a_{ijk} \in \C$.\\
 \a \textbf{Output}: State if $A$ is degenerate, determine $F_{ess}(A)$, and compute $Trk\, A$.

\a \begin{boxedminipage}{135mm}
    \begin{algorithmic}[1]

	\STATE\label{Alg3Step1} {\it STEP 1)} \  Compute $\det N$:
	
	- If $r=4$ and $\bN$ is a smooth quadric, then $A$ is non-degenerate, concise and $Trk\, A=4$. {\bf STOP}
	
	- If  $r\geq 5$ and $\bN$ is a cone on a smooth quadric in $\P^3$, then $A$ is non-degenerate, $Trk\, A=4$ and $A$ is not concise, with $F_{ess}=(2,2,4)$ {\bf STOP}

- If $\bN$ is not as in the previous cases, $A$ is not concise; GO TO NEXT STEP.
	
	   \STATE\label{Alg3Step2}{\it STEP 2)} \  If $\bN$ is a cone on a smooth conic in $\P^2$  (i.e. the vertex of the cone is a linear space of dimension $r-4$ in $\P^{r-1}$), then $A$ is  non-degenerate, $F_{ess}(A)=(2,2,3)$ and $Trk\, A=3$  {\bf STOP}

Otherwise $A$ is degenerate; GO TO NEXT STEP.

\STATE\label{Alg3Step3}{\it STEP 3)}  If $\bN$ is a cone on a conic in $\P^2$ made by two distinct lines (i.e. $\bN$ is made of two distinct hyperplanes) compute the minors of $L$; then:	 
	   
	   - If  $\bL$ is a simple point in $\P^1$, then $F_{ess}(A)=(2,2,3)$  and $Trk\, A=3$. {\bf STOP}

         - If  $\bL$ is made by two simple points in $\P^1$, then $F_{ess}(A)=(2,2,2)$  and $Trk\, A=2$. {\bf STOP}

 - If $\bN$ is not as above GO TO NEXT STEP.
		
			\STATE\label{Alg3Step3}{\it STEP 4)} \ If $\bN$ is a double hyperplane, compute the minors of $L$; then:
	
	- If the minors of $L$ define a double point in $\P^1$, then $F_{ess}(A) = (2,2,2)$ and $Trk\, A=3$. {\bf STOP}
		
		 - If the minors of $L$ vanish identically, i.e. $\bL=\P^1$, then  $F_{ess} = (2,2,1)$ and $Trk\, A=2$. {\bf STOP}
	  
	 - If $\det N$ vanishes identically, GO TO NEXT STEP.
	
	 	 		\STATE\label{Alg3Step4}{\it STEP 5)} \ When $\det N$ vanishes identically, i.e. $\bN=\P^{r-1}$, compute the minors of $M$ and of $L$; then:	 
	   
	   - If one between $\bL$ and $\bM$ is a double point and the other is all of $\P^1$, then $F_{ess} = (1,2,2)$ or $F_{ess} = (2,1,2)$ and $Trk\, A=2$. {\bf STOP}
	  
	 - If $\bL=\bM=\P^1$, i.e. all the minors of the two matrices are identically 0, then $Trk\, A= 1$ and $F_{ess}(A)=(1,1,1)$. {\bf STOP}
	  
  \end{algorithmic}
\end{boxedminipage}
\end{algorithm}

\bigskip
\bigskip
The first author is a member of INdAM-GNSAGA.


\begin{thebibliography}{biblio}
\bibitem [A1]{A1} S.Abrescia, {\it Tensori tridimensionali e varietà determinantali} Ph.D. Thesis, Univ. of Bologna. (2003)
\bibitem [A2]{A2}  S.Abrescia, {\it A characterization of degenerate tridimensional
tensors}. Journal of Pure and Applied Algebra 187 (2004) 1 – 17
\bibitem [ACGH]{ACGH} M.Arbarello, P.Cornalba,  A.Griffiths, J.Harris, {\it Geometry of Algebraic Curves}, 1985, Springer New York, NY
\bibitem[BCCGO]{BCCGO} A.Bernardi, E. Carlini, M.V.Catalisano, A. Gimigliano, A. Oneto  {\it The hitchhiker guide to secant varieties and tensor decomposition}. Mathematics 6 (2018,  special issue: "Decomposibility of Tensors", L.Chiantini Ed.);  DOI: 10.3390/math6120314. Also available at    http://www.mdpi.com/2227-7390/6/12/314/pdf
\bibitem [BV] {BV} W. Bruns, U. Vetter,  {\it Determinantal Rings}, Lecture Notes in Mathematics,  1327 , Springer, Berlin, Heidelberg  1988.
\bibitem [Ca] {Ca} E. Carlini. {\it Reducing the number of variables of a polynomial.} In Algebraic geometry and geometric modeling, 237–247. Springer, 2006.
\bibitem [CGG1] {CGG1} M.V.Catalisano, A.V. Geramita, A. Gimigliano, {\it Tensor rank, secant varieties to Segre varieties and fat points in multiprojective spaces} Queen's Papers in pure and Applied Math.  XIII (2001), 225-246.
\bibitem [CGG2] {CGG2} M.V.Catalisano, A.V. Geramita, A. Gimigliano, {\it Rank of tensors, secant varieties to Segre varieties and fat points} Linear Algebra and its Appl. 355 (2002), 263-285.
\bibitem [CoCoA] {CoCoA} J. Abbott, A. M. Bigatti, L. Robbiano, {\it CoCoA: a system for doing Computations in Commutative Algebra}.
Available at https://sites.google.com/view/cocoa-c
\bibitem [CS] {CS} G.Comas, M. Seiguer, {\it On the Rank of a Binary Form}. Found. Comput. Math 11, (2011) 65–78. https://doi.org/10.1007/s10208-010-9077-x
\bibitem [E] {E} R. Ehrenborg {\it Canonical Forms of Two by Two by Two Matrices} J. of Algebra, 213 (1999), 195-224.
\bibitem  [F] {F}  S. Friedland, {\it On the generic and typical ranks of 3-tensors}.
Linear Algebra Appl. 436 (2012) 478–497.
\bibitem[GKZ] {GKZ} I.M. Gelfand, M.M. Kapranov, A. Zelevinsky, {\it Discriminants, Resultants and Multidimensional Determinants}, Birchauser, Basel, 1994.
\bibitem[Gh]  {Gh}  F. Gherardelli, {\it Osservazioni sugli iperdeterminanti}, Istituto Lombardo (Rend. Sc.) A 127, (1993) 107-113.
\bibitem [K1] {K1} J.B. Kruskal, {\it Three-way arrays: rank and uniqueness of trilinear decompositions, with applications to arithmetic complexity
and statistics}, Linear Algebra Appl. 18 (1977) 95–138.
\bibitem [K2] {K2} J.B. Kruskal, {\it Rank, decomposition, and uniqueness for 3-way and N-way arrays}, Multi-way Data Analysis, vol. 7–18, North-Holland, Amsterdam, 1989.
\bibitem [Oed] {Oed} L. Oeding,  {\it The hyperdeterminant of a symmetric tensor}, Advances in Math., 231 (2012), 1308-1326.
\bibitem [Ol] {Ol} R. Oldenburger, {\it On canonical binary trilinear forms}, Bull. of the AMS,  38 (1932), 385-387.
\bibitem [O] {O}  G. Ottaviani, {\it Introduction to the Hyperdeterminant and to the Rank of Multidimensional Matrices}. In: Peeva, I. (eds) {\it Commutative Algebra} (2013). Springer, New York, NY. DOI:10.1007/978-1-4614-5292-8-20.
\bibitem [Sc] {Sc} L. Schläfli, {\it  Uber die Resultante eines Systems mehrerer algebraischer Gleichungen}, Denkschr. der Kaiserl. Akad. Wiss., Math-Naturwiss. Klasse, 4(1852),
reprinted in: {\it Gessamelte Abhandlungen}, vol.2, N0 9, p. 9-112, Birkhauser-Verlag,
Basel, 1953. 
\bibitem [TC] {TC} R.M. Thrall, H. Chanler, {\it Ternary trilinear forms in the field of complex numbers}, Duke Math. J., 4,  (1938), 678-690
\bibitem [T] {T}  R. M. Thrall, {\it On Projective Equivalence of Trilinear Forms}
Annals of Mathematics, 42 (1941), 469-485



\end{thebibliography}
\end{document}